\documentclass[11pt,a4paper]{amsart}
\usepackage[english]{certus}

\usepackage{graphicx}
\usepackage{graphics}
\usepackage{pict2e}
\usepackage{epic}
\usepackage{eqnarray}
\usepackage{varwidth}
 \usepackage{enumerate}
\usepackage{pdfpages}
\usepackage{blindtext}
\usepackage{epstopdf}

\DeclareMathOperator{\stab}{stab}

\newcommand*{\Z}{\mathbb Z}
\newcommand*{\R}{\mathbb R}
\newcommand*{\N}{\mathbb N}

\theoremstyle{plain}
\newtheorem{thm}{Theorem}[section]
\newtheorem{lemma}[thm]{Lemma}
\newtheorem{cor}[thm]{Corollary}
\newtheorem{prop}[thm]{Proposition}

\theoremstyle{definition}
\newtheorem{defi}[thm]{Definition}

\theoremstyle{remark}
\newtheorem{rem}[thm]{Remark}
\newtheorem{ex}[thm]{Example}
\newtheorem{qu}[thm]{Question}

\theoremstyle{plain}

\title{Sofic boundaries and a-T-menability}
\author{Vadim Alekseev} 
\address{Vadim Alekseev, Technische Universit\"{a}t Dresden, Fakultät Mathematik, Institut f\"{u}r Geometrie, 01062, Dresden, Deutschland}
\email{vadim.alekseev@tu-dresden.de}
\author{Leonardo Biz}
\address{Leonardo Biz, Technische Universit\"{a}t Dresden, Fakultät Mathematik, Institut f\"{u}r Geometrie, 01062, Dresden, Deutschland}
\email{leonardo.businhani\_biz@tu-dresden.de}


\begin{document}
\onehalfspace		
	\begin{abstract} 
	We undertake a systematic study of the approximation properties of the topological and measurable versions of the coarse boundary groupoid associated to a sequence of finite graphs of bounded degree. On the topological side, we prove that asymptotic coarse embeddability of the graph sequence into a Hilbert space is equivalent to the coarse boundary groupoid being topologically a-T-menable, thus answering a question by Rufus Willett. On the measure-theoretic side, we prove that measure-theoretic amenability and a-T-menability of the coarse boundary groupoid are related to hyperfiniteness and property almost-A resp. an version of ``almost asymptotic embeddability into Hilbert space''. These results can be directly applied to spaces of graphs coming from sofic approximations.
	\end{abstract} 
	
		\maketitle
	\tableofcontents
	
	

	\tableofcontents

	\section{Introduction}
	
	This paper continues the previous work by first named author and Martin Finn-Sell from \cite{alekseev2016sofic}, relating coarse geometry of the graph spaces obtained from sofic approximations of a group $\Gamma$ and analytic properties of the group, using the coarse boundary groupoid to connect them. There, it was proven that coarse-geometric properties of the sofic approximation like property A or asymptotic coarse embeddability into a Hilbert space imply amenability resp. a-T-menability of the group; that left the question which properties of the sofic approximation would be \emph{equivalent} to amenability resp. a-T-menability of the group. 

	Since then, important progress has been made in the work of Tom Kaiser \cite{kaiser2019combinatorial}, introducing property almost-A of a graph sequence which, applied to a sofic approximation, indeed characterises amenability of a sofic group; however, the techniques used there were local in nature and it remained a task to connect them in the general technique of coarse groupoids, possibly also obtaining the right technology to attack the problem in the a-T-menable case.

	In this paper we systematically investigate both amenability and a-T-menability of the coarse boundary groupoid $\partial G(X)$  attached to a sequence of bounded degree graphs $\{X_i\}_{i\in\N}$. We make use of the idea from \cite{alekseev2016sofic} that the coarse boundary groupoid can usefully be considered both as topological and measured groupoid, reflecting the difference between more ``rigid' coarse geometric properties and ``geometry up to negligible subsets'', usually encountered in the context of sofic approximations.

	In the amenable case, we prove the following result.
	
	\begin{ThmA}[Theorem \ref{mainthm}]
		Let $\mathcal{X}=\{X_i\}_{i\in\N}$ be a sequence of finite graphs with bounded degree such that the cardinality of $X_i$ goes to infinity when $i$ goes to infinity. Take $X$ to be the space of graphs of this sequence and ${\partial G(X)}$ the related coarse boundary groupoid. The following statements are equivalent:
		\begin{enumerate}[$(1)$]
			\item The measurable coarse boundary groupoid $({\partial G(X)},\hat{\mu}_\omega)$ is measurably amenable, for every non-principal ultrafilter $\omega\in\partial\beta\N$;
			\item  The sequence of graphs $\mathcal{X}$ has property A on average along every non-principal ultrafilter $\omega\in\partial\beta\N$;
			\item The sequence of graphs $\mathcal{X}$ has property almost-A along every non-principal ultrafilter $\omega \in \partial \beta \N$;
			\item The sequence of graphs $\mathcal{X}$ is hyperfinite;
			\item The measurable equivalence relation $({\partial G(X)},\hat{\mu}_\omega)$ given by the coarse boundary groupoid is $\mu_\omega$-hyperfinite, for every non-principal ultrafilter $\omega\in\partial\beta\N$.
		\end{enumerate}
	\end{ThmA}
	
	The main novelty of this theorem is to reinterpret the known results about hyperfiniteness and property almost-A using the language of groupoids. Indeed, as it turns out, the subtle difference between property A and hyperfiniteness turns out to be precisely the difference between topological and measured amenability of the coarse boundary groupoid.
	
	This analogy also helps us to deal with the a-T-menable case. On the topological side, we were able to answer a question raised by Rufus Willett in \cite{willett2015random} where he proved that \emph{asymptotic coarse embeddability} is a sufficient condition for the a-T-menability of the coarse boundary groupoid $\partial G(X)$, when $X$ is the space of graphs of a sequence of finite bounded degree graphs. We were able to prove the converse of Willett's result:

	\begin{ThmA}[{Theorem \ref{thmtopaTmenable}}]
		Let $\mathcal{X}=\{X_i\}_{i\in\N}$ be a sequence of finite graphs with bounded degree and $X$ the related space of graphs. If the coarse boundary groupoid $\partial G(X)$ is topologically a-T-menable then the sequence of graphs $\mathcal{X}$ is asymptotically coarsely embeddable into a Hilbert space $\mathcal{H}.$
	\end{ThmA}


	Analogously to the amenable case, the above statement fails if the groupoid is suppposed to only be measurably a-T-menable. However, similarly to the amenable case, it becomes asymptotically coarsely embeddable with we remove a set of small measure along the sequence of graphs, one important difference being that we need to consider the old graph metric.

	\begin{ThmA}[{Theorem \ref{thmMeaaTmenable}}]
		Let $\mathcal{X}=\{X_i\}_{i\in\N}$ be a sequence of finite graphs with bounded degree and $X$ the related space of graphs. If the coarse boundary groupoid $(\partial G(X),\mu_\omega)$ is a measurably a-T-menable for every non-principal ultrafilter $\omega\in\partial\beta\N$, then, for every $\varepsilon>0$, there exist $\{Z_i\}_{i\in\N}\subset \{X_i\}_{i\in\N}$ such that $\lim\limits_{i\to\infty} {\mu}_i(Z_i)\leq \varepsilon$ and $\{X_i\backslash Z_i\}_{i\in\N}$ is asymptotically coarsely embeddable into a Hilbert space $\mathcal{H}$ when equipped with the same metric on $X$.
	\end{ThmA}

	
	Findally, we notice that in the sofic case the above statement can be improved to yield $\varepsilon=0$:

	\begin{ThmA}[{Proposition \ref{aTmeansoficapp}}]
		Let $\Gamma$ be a sofic finitely generated group with sofic approximation $\mathcal{G}=\{X_i\}_{i\in\N}$. Then, the group $\Gamma$ is a-T-menable if and only if there is a sequence of subgraphs $\{Z_i\}_{i\in\N} \subset\{X_i\}_{i\in\N} $ of the sofic approximation $\mathcal{G}$ such that $\lim\limits_{i\to \infty} \frac{|Z_i|}{|X_i|}=0$ and $\{X_i\setminus Z_i\}_{i\in\N} $ is asymptotically coarsely embeddable into a Hilbert space $\mathcal{H}$ with the old metric of the sofic approximation.
	\end{ThmA}

	\textbf{Acknowledgements.} The authors would like to thank Rufus Willett for comments and suggestions that greatly helped to improve the text, Martin Finn-Sell and Andreas Thom for stimulating discussions.

	\section{Coarse geometry}
	
	Coarse geometry is the study of spaces from a far away point of view. Different from what we learn in Analysis that they are interested in an small scale properties of the space, here we care about when bounded spaces are preserved by maps. In this large scale geometry, the integers set is equivalent to the real numbers, any finite set is the same as a point. This point of view is proper to work with infinite objects, their approximations and some related geometric properties that we define soon. For an introduction into this world, see \cite{roe2003lectures}.

	We start this subsection remarking that all spaces we work with are metrizable. In particular, we are also considering that they are uniformly discrete metric space, i.e, there is a $\delta>0$ such that $d(x,y)>\delta$, for all $x\neq y$ in the metric space $(X,d)$. 
	
	We are interested in the large scale properties of a sequence of graphs. Given a sequence of finite metric spaces $(X_i,d_i)_{i\in\N}$, we can consider it as a single metric space and verify some coarse geometric properties along them. For that we need a metric compatible with the previous distance $d_i$ when restricted to the respective $X_i$ and is far apart on distinct $X_i$'s. Formally, we have:  
	
	\begin{defi}
		Given a sequence of finite metric spaces $(X_i,d_i)_{i\in\N}$. The \emph{coarse union} of the sequence is a metric space $X=(X,d)$ where $X = \sqcup_{i\in \N} X_i$ is the disjoint union of the given metric spaces equipped with a metric $d$ that satisfies:
		\begin{enumerate}[$(i)$]
			\item  the distance $d\big|_{X_i}$ coincides with the metric on $(X_i,d_i)$, for all $i\in\N$
			\item $d(X_i,X_j)$ goes to infinity when $i+j$ goes to infinity.
		\end{enumerate} 
	\end{defi}
	
	Such a metrization of the coarse union always exists. Indeed, some authors even change the second item for 
	
	\begin{enumerate}[$(ii')$]
		\item $ d(X_i,X_j)=\diam(X_i) + \diam(X_j) +i+j$, 
	\end{enumerate} 
	to make explicit the distance between the spaces. Any metric that satisfies $(ii)$ or $(ii')$ are coarse equivalent. Even more, all metrization as above are coarse equivalent, that is, given a sequence of metric spaces $\{X_i\}_{i\in\N}$ such that the coarse union admits two distinct metrization $d_1$ and $d_2$ that satisfies $(i)$ and $(ii)$ of the previous definition. Then $(\sqcup_{i\in \N} X_i,d_1)$ is coarsely equivalent to $(\sqcup_{i\in \N} X_i,d_2)$.

	More that a sequence of metric spaces, we are particularly interested in a sequence of graphs and their coarse properties. We fix here the coarse union of specific graphs that we will use along this work.
	
	\begin{defi}
		For a given sequence of finite graphs with bounded degree $\{X_i\}_{i\in \N}$ such that the cardinality of $X_i$ is going to infinity when $i$ grows, we equip each graph with the length distance and we call \emph{space of graphs}, the coarse union $X$ of those graphs with the metric $d$ that satisfies item $(i)$ and $(ii)$ of the previous definition.
	\end{defi}
	
	Sometimes we call it space of bounded degree finite graphs to emphasize those graph properties that we fixed in the beginning. We apply this construction mainly to box spaces and sofic approximation. By an abuse of notation we are still calling the space of graphs of them as box space and sofic approximation. 
	
	\begin{ex}
 Given a finitely generated residually finite graph $\Gamma = \langle S \rangle$, by the residually finitennes of the group, there exits a sequence of normal subgroups $\{\Lambda_i\}_{i\in\N}$ of $\Gamma$. It is known that each Cayley graph $\Cay(\Gamma / \Lambda_i, q(S))$ related with the filtration $\{\Lambda_i\}_{i\in\N}$ have the length metric of the graphs $\Cay(\Gamma / \Lambda_i, q(S))$, where $q$ is the quotient map. We now denoted by $(\square_{\Lambda_i} \Gamma,d)$ the \emph{box space of $\Gamma$}, where $(\square_{\Lambda_i} \Gamma)$ is the coarse union as above.  Notice that it is independent on the choice of generators for the group since both are coarse equivalent.
	\end{ex}

The coarse union of a sofic approximation it will only appear in the last section, where we apply the important results to this specific case.

	What is important in this metrization $(X,d)$ of the coarse union is that for a given $R>0$, the cardinality of the balls of radius $R$ centred in a fixed element $x \in X$ is finite. A space that hold this property is called \emph{bounded geometry space}. This is the case of the space of graphs.

	\subsection{Property A or coarse amenability}

	Property A was introduced by Yu in \cite{yu2000coarse} as a coarse version of amenability. That implies that it is also a weak form of amenability. In the same article, Yu proves that metric spaces with property A are coarsely embeddable into a Hilbert spaces, in particular they also satisfy the Coarse Baum-Connes Conjecture. These results are some motivation to study the coarse amenability. A great survey in this area is \cite{willett2006some}, where one can find many equivalent definitions and most of the results we remarked here.

	\begin{defi}\label{propA}
		Let $\mathcal{X}=\{X_i\}_{i \in \N}$ be a sequence of finite graphs and $X$ the related space of graphs.
		The sequence of graphs $\mathcal{X}=\{X_i\}_{i \in \N}$ has \textit{property A}, if for every $\varepsilon>0$, there is a constant $S>0$(not depending on $i$) and a function $\eta: X \to \ell^1(X) $ mapping $x \mapsto \eta_x$
		such that:
		\begin{enumerate}[$(i)$]
			\item $||\eta_x ||_1=1$, for all $x \in X$;
			\item For all $x\in X$, $\eta_x$ is supported in a ball of radius $S$ around $x$, i.e, $\supp(\eta_x) \subseteq B_S(x)$;
			\item If $(x,y)\in E(X_i)$, then $||\eta_x - \eta_y ||_1< \varepsilon$.
			
		\end{enumerate}
	\end{defi}

	Let's see some classical examples and non-examples.

	\begin{ex}
		\begin{enumerate}
			\item  Let $\Gamma$ be a finitely generated residually finite group. The group $\Gamma$ is amenable if and only if $\square_f\Gamma$ has property A (Theorem 11.39 of \cite{roe2003lectures}), where the subindex $f$ on the box spaces means that it is the box space related with the filtration containing all normal subgroups. This box space is also named as the \emph{full box space of the group}.
			
			\item The most famous non-examples of sequence of graphs without property A are the sequence of expanders graphs, introduced in \cite{lubotzky1988ramanujan}. Roughly speaking, a sequence of expanders is a sequence of highly connected graphs with bounded degree. It can be formally defined in different ways. So, a sequence of graphs $\{X_i\}_{i\in\N}$ with bounded degree such that $|X_i|$ goes to infinity when $i$ grows is a sequence of \emph{expanders} if there exist a constant $C>0$ such that for all $i\in \N$ and every 1-Lipschitz map $\psi: X_i \to \ell^2(X_i)$ satisfies 
				$$ \sum\limits_{x,y \in X_i} ||\psi(x)-\psi(y) ||_2^2 \leq C |X_i|^2 . $$

			\item All sequence of finite graphs $\{X_i\}_{i\in\N}$ with degree bounded $d$ with $d>3$ and girth (the length of the shortest cycle) tending to infinity does not have property A. This class of examples were proved by Rufus Willett in \cite[Theorem 1.2]{willett2011property}.

		\end{enumerate}
	\end{ex}

	\subsection{Coarse embeddability into a Hilbert space}
	
	 For geometric reasons, we want to consider coarse embeddings into a Hilbert space $\mathcal{H}$, i.e, a completed normed space where the norm was induced by an inner product. There exist a more classical definition mapping the space into a Hilbert space where the map should satisfies few properties but here, we will take the definition that use kernels.

	Let $X$ be a set. A \emph{kernel} $K:X\times X \to \R$ is a symmetric map, i.e, $K(x,y)=K(y,x)$, for all $x,y \in X$.  The normalization of the kernel depends if it is of positive or negative type.

	\begin{defi}
		A kernel $K:X\times X \to \R$ is \emph{conditionally negative definite} if for all sequence $x_1,\cdots,x_n \in X$ and real numbers $\lambda_1,\cdots,\lambda_n$ such that $\sum\limits_{j=1}^{n} \lambda_j =0$ we have, $$\sum_{i,j=1}^{n} \lambda_i \lambda_j K(x_i,x_j) \leq 0 .$$ 

		Even more, for $K$ a conditionally negative definite kernel, we say that $K$ is \emph{normalized} if $K(x,x)=0$, for all $x\in X$.
	\end{defi}

	Negative and positive type is not only the matter of a negative sign because there are maps that hold both properties. On the other hand, it is clear to notice that a constant kernel is conditionally negative definite. Notice that conditionally negative definite kernels are closed under addition, scalar multiplication by a positive number and pointwise limit.

	The following example is considered the connection between kernels and embeddability into a Hilbert spaces. Let $f$ be a map from a metric space $(X,d)$ to some Hilbert space $\mathcal{H}$. Then, $K(x,y)=|| f(x)-f(y)||_{\mathcal{H}}^2$ defines a kernel of conditionally negative type. Via GNS type construction, given a symmetric normalized kernel $K:X\times X\to \R$ of conditionally negative type, one constructs a Hilbert space $\mathcal{H}$ and a map $f:X\to \mathcal{H}$ such that $K(x,y)=|| f(x)-f(y)||^2$. Thus, we can work with maps into a Hilbert space using kernels.
	
	\begin{defi}
		Let $(X,d)$ be a metric space. We say that $(X,d)$ admits a \emph{coarse embedding into a Hilbert space} if there is a normalized conditionally negative definite kernel $K:X\times X \to \R$ and maps $\rho_1,\rho_2:\R_+ \to \R$ such that $\rho_i(R)$
			go to infinity when $R$ grows and 
			$$ \rho_1(d(x,y)) \leq K(x,y)\leq \rho_2(d(x,y)). $$
			Those maps satisfying the above inequality are called \textit{control functions}. Sometimes we use the notation $\lim\limits_{r\to \infty} \rho_i(r) = \infty$, to say that $\rho_i$ goes to infinity at infinity, for $i=1,2$. 
	\end{defi}


	\subsection{Weakenings of coarse embeddability}

	 In this section, we finally define it and study some variations of embeddability of metric spaces. We ask a bit of patience from the reader. The motivation to define weak forms of embeddability will be clarified in this work.
	
	The first weak version is called \emph{fibred coarse embeddability}, the idea here is to embed the space in a family of Hilbert spaces, i.e, each point can be sent in a different Hilbert space. Since we would like keep some geometric structures, we need maps giving a certain compatibility between the Hilbert spaces. This notion was first defined in \cite{chen2013maximal} by Xiaoman Chen, Qin Wang and Guoliang Yu.

	Fibred embeddability is not the focus of our work, but there is a direct relation between box spaces of groups holding this property and the a-T-menability of the groups, as we will see on the next section. Notice that if a metric space $(X,d)$ is coarsely embeddable into a Hilbert space $\mathcal{H}$ then it is fibred coarsely embeddable, for that just take $\mathcal{H}_x= \mathcal{H}$. However, fibred coarse embeddability is much weaker, in the sense that some sequence of expanders are fibred coarsely embeddable into a Hilbert space.
	
	
	In \cite{willett2015random}, Willett introduced another weak version of coarse embeddability for a sequence of graphs, called \emph{asymptotic coarse embeddability}. This property is also weaker than the fibred coarse embeddability. In the case of the spaces of graphs, coarse embeddabilitty into $\mathcal{H}$ implies the fibred coarse embeddabillity (the converse doesn't hold as witnessed by expanders); in turn, fibred coarse embeddability also implies the asymptotic coarse embeddabillity into a Hilbert space.
	
	The notion of asymptotic coarse embeddability ask for a weakly form of conditionally negative definiteness on the sequence of kernels $(K_i)_{i\in\N}$ defined in each graph $X_i$. Basically, the a sequence of graphs $\{X_i\}_{i\in \N}$ is linked with a sequence of real numbers $(R_i)_{i\in \N}$ going to infinity such that our sequence of kernels is ``$(R_i)$-locally''  conditionally negative definite, i.e, each kernel $K_i$ is conditionally negative on finite sets $\{x_1, \dots, x_n \}\subseteq X_i$ with diameter smaller than $2R_i$. More precisely: 
	
	\begin{defi}\label{asympEmb}
		Let $\mathcal{X}=\{X_i\}_{i\in \N} $ be a sequence of finite graphs of bounded degree such that the cardinality of the graphs $X_i$ goes to infinity when $i$ grows. We say that $\mathcal{X}$ admits an \textit{asymptotically coarse embedding into a Hilbert space} $\mathcal{H}$ if, there is a sequence of symmetric kernels $K_i: X_i \times X_i \to \R $ with non-decreasing control functions $\rho_1,\rho_2:\R_+ \to \R $ tending to infinity at infinity, and a sequence of non-negative real numbers $(R_i)_{i\in \N}$ going to infinity such that for all $i\in\N$:
		\begin{enumerate}[$(1)$]
			
			\item The sequence of kernels $(K_i)_{i\in\N}$ is  \textit{limit-normalized}, i.e, for any sequence $(x_i)_{i\in\N}$ of points in $X$ such that $x_i \in X_i$, the sequence $K_i(x_i,x_i)$ tends to zero, when $i$ goes to infinity. 
			\item For all $ x,y \in X_i$,
			$$ \rho_1(d(x,y)) \leq K_i(x,y) \leq \rho_2(d(x,y)).$$ 
			
			\item For any subset $\{x_1,\dots,x_n\}\subseteq X_i$ of diameter smaller than $2 R_i$ and any collection of real numbers $\lambda_1,\dots,\lambda_n$ with $\sum\limits_{k=1}^n \lambda_k=0$, we have $$\sum\limits_{k,l=1}^{n} \lambda_k\lambda_l K_i(x_k,x_l) \leq 0 .$$  
		\end{enumerate}
		
		For a sequence of kernels satisfying the item $(3)$ above for a sequence $(R_i)_{i\in \N}$, we called it \textit{$(R_i)$-locally conditionally negative definite kernels}.
	\end{defi}


	
	Notice that we slightly change the definition given by Willett. Here, we have a weaker condition than the negative type normalization ($K_i(x,x)=0$, for all  $x\in X_i$). Under this new condition we still have a bijection between symmetric normalized conditionally negative type kernels $K:X\times X\to \R$ and maps $f:X\to \mathcal{H}$, but we lost the identity $K(x,y)= ||f(x)-f(y) ||^2_{\mathcal{H}}$. Indeed, given a symmetric limit-normalized conditionally negative type kernel $K:X\times X\to \R$, set $C^{(0)}_c(X)$ as the vector space of finite supported functions $f:X\to \R$ such that $\sum\limits_{x \in X}f(x)=0$. We can define the bilinear form in $C^{(0)}_c(X)$ as
	
	$$\langle f,g\rangle= -\frac{1}{2}  \sum_{x,y \in X} K(x,y)f(x)g(y).$$
	
	That is symmetric because $K$ it is and $\langle f,f\rangle\geq 0$ for all $f\in C^{(0)}_c(X)$, by the negative typeness of the kernel. Even more, this inner product satisfies the Cauchy-Schwartz inequality, since $\langle tf+g,tf+g\rangle \geq 0$ for all $t\in\R$ and $f,g\in C^{(0)}_c(X)$. So, $t^2\langle f,f\rangle +2t \langle f,g\rangle + \langle g,g \rangle \geq 0$. Thus, the discriminant $4 (\langle f,g\rangle)^2 - 4 \langle f,f\rangle \langle g,g\rangle $ is non-positive. Thus, it satisfies the Cauchy-Schwartz inequality $ |\langle f,g\rangle|^2 \leq |\langle f,f\rangle| \cdot |\langle g,g\rangle| $. The separated completion $\mathcal{H}$ of $C^{(0)}_c(X)$ the induced inner product is a Hilbert space.
	
	For a fixed point $x_0 \in X$, we define $f: X \to \mathcal{H}$ as $x\mapsto \delta_x-\delta_{x_0}$. But then,
	$$||f(x)-f(y) ||_{\mathcal{H}}^2=\langle \delta_x-\delta_y,\delta_x-\delta_y \rangle = K(x,y) -\dfrac{1}{2} K(x,x)-\dfrac{1}{2} K(y,y).$$ 
	We do not necessarily have $||f(x)-f(y) ||_{\mathcal{H}}^2= K(x,y)$, because $K$ is not normalized everywhere, but only along the limit; however, the existence of $f$ is enough for geometric purposes. For this reason, we still using the same name ``asymptotic coarse embedding'' as given by Willett in \cite{willett2015random}.

	\subsection{Summary of known results on spaces of graphs}

	In this section we give an overview of some results that link properties of the group (amenability and a-T-menability) with the geometric properties on the sequence of graphs (hyperfiniteness, property A and embeddability into a Hilbert space). The results presented here are taken from \cite{alekseev2016sofic},\cite{elek2006combinatorial},\cite{finn2014fibred},\cite{roe2003lectures} and \cite{willett2006some}.

	We start with the notion of hyperfiniteness for a sequence of graphs.
		\begin{defi}
		Let $\mathcal{X} = \{X_i\}_{i \in\N}$ be a sequence of finite graphs. The sequence $\{X_i\}_{i \in\N}$ is called \textit{hyperfinite} if for all $\varepsilon >0$, there is a $K_{\varepsilon} \in \N$ and a partition in connected components of the vertex sets $V(X_i)= A_1^i \cup A_2^i \cup \dots \cup A_{n_i}^i$ such that
		\begin{itemize}
			\item $|A_j^i|< K_\varepsilon$, for all $i \in \N$ and $1\leq j \leq n_i$;
			\item If $E_i^\varepsilon = \{(x,y)\in E(X_i); x \in A_n^i, y \in A_m^i, n\neq m\}$ is the set of edges connecting two different components then
			$$ \limsup_{i \rightarrow \infty} \dfrac{|E_i^\varepsilon|}{|V(X_i)|} \leq \varepsilon .$$
		\end{itemize}
	\end{defi}

	The idea of hyperfiniteness is that we can remove the edges in the set $E_i^\varepsilon$ of the sequence of graphs such that the remaining sequence of graphs $\mathcal{X}^\prime$ consists of connected components with size of at most $K_\varepsilon$. 

	For this section we fix a finitely generated residually finite group $\Gamma$ with a filtration of normal subgroups $\{\Lambda_i\}_{i\in \N }$. In a very combinatorial way, Elek proved in \cite{elek2006combinatorial} the following result. 
	
	\begin{prop}[{\cite[Proposition 1.4]{elek2006combinatorial}}]
		The box space $\square_{\Lambda_i} \Gamma$  is hyperfinite if and only if $\Gamma$ is amenable. 	
	\end{prop}

\subsubsection{Coarse properties and box spaces} 
	
 The following classical result relating box spaces and approximation properties is due to John Roe and Erik Guentner. Here  $\square_{\Lambda_i} \Gamma$ is the box space of $\Gamma$ with respect to a filtration $\{{\Lambda_i}\}_{i\in\N}$.
	
		\begin{prop}[{\cite[Proposition 6.3.3 and 6.3.4]{willett2006some}}]\label{boxspaceRoe}
			Let $\Gamma$ be finitely generated residually finite group.
			\begin{enumerate}[$(1)$]
				\item A finitely generated residually finite group $\Gamma$ is amenable if and only if the box space $\square_{\Lambda_i}\Gamma$ has Property A.
				\item 	If the box space $\square_{\Lambda_i} \Gamma$ is coarsely embeddable into a Hilbert space then $\Gamma$ is a-T-menable.
			\end{enumerate}
		\end{prop}



		

	It is known that a-T-menability is not enough to imply coarse embeddability. There exist certain filtrations of the free group such that the associated box space is a sequence of expanders, so it can do not be coarsely embeddable into a Hilbert space. A good overview of this construction can be found on \cite{khukhro2014embeddable}. Even more, one can see in \cite{willett2012higher}, that $\square \SL_3(\Z)$ is not even fibred coarsely embeddable into a Hilbert space. On the other hand, a-T-menability is sufficient to imply the fibred coarse embeddability of the box space.

	\begin{prop}[{see \cite{chen2013characterization}, \cite{chen2013maximal} and \cite{finn2014fibred}}]
		The box space $\square_{\Lambda_i} \Gamma$ is fibred coarsely embeddable into a Hilbert space if and only if $\Gamma$ is a-T-menable.
	\end{prop}
	


	\subsubsection{Coarse properties and sofic approximations}
	
	 In this section we see which results of the last section can be generalized for a sofic approximation. Before showing the results, we would to point some basic information about soficity that will be use later.

	 First introduced by Gromov in \cite{gromov1999endomorphisms} as a generalization of amenability and residually finiteness and named by Weiss in \cite{weiss2000sofic} sofic groups become an important object of study in the past decades. 
	 
	 There are many ways to define the soficity of a group. Here we first take the characterization of a sofic group via the almost actions and after we see the graph version of it. A good survey on this topic is \cite{pestov2008hyperlinear}.

	  \begin{defi}\label{sofic}
	 	A discrete group $\Gamma$ is \textit{sofic} if, for every finite subset $F\subset\Gamma$ and all $\varepsilon>0$, there is a finite set $X$, a map $\sigma:\Gamma\to \Sym(X)$ and a subset $Y\subset X$ with $|Y|\geq (1-\varepsilon)|X|$ such that 
	 	\begin{align}
	 		\sigma(g)\sigma(h)(y)&=\sigma(gh)(y),  &\forall g,h \in F, y\in Y \\
	 		\sigma(g)(y)&\neq y, & \forall g \in F\backslash \{e\}, y\in Y.
	 	\end{align}
	 \end{defi}
	 
	 A map $\sigma$ that satisfies $(2.1)$ and $(2.2)$ as above is called \textit{$(F,\varepsilon)-$injective almost action}. A sequence of finite sets $X_i$ with an $(F_i,\varepsilon_i)-$injective almost action such that $F_i$ is a nested sequence of subsets that exhaust an infinte group $\Gamma$ and $\varepsilon_i$ tends to zero, when $i$ goes to infinity, is called a \textit{sofic approximation} of $\Gamma$. Notice that we are considering infinite groups, thus the cardinality of $X_i$ is growing along the sequence.

	 In the particular case where $\Gamma$ is a finitely generated sofic group with finite generating set $S$, we can construct an $S$-labeled graph using the sets $X_i$ of the sofic approximation as the vertex set and define the edges as the set such that every $x \in X_i $ is connected by the respective $\sigma_i(s)$ to $\sigma_i(s)x$, for all $s\in S$. Abusing the notation, we also call this sequence of graphs \textit{sofic approximation}. Moreover, by \cite[Lemma 2.19]{alekseev2016sofic}, we can assume that those graphs are connected for each $i\in\N$. With this construction, one obtains the following alternative definition of soficity:
	 
	 \begin{defi}
	 	Let  $\Gamma$ be a group generated by a finite set $S$. We said that $\Gamma$ is \textit{sofic} if there is a sequence $\{X_i\}_{i \in \N}$ of bounded degree finite $S$-labelled graphs such that $\{X_i\}_{i \in \N}$ Benjamini-Schramm converges to $(\Cay(\Gamma,S),e)$, where a sequence of graphs $\{X_i\}_{i \in \N}$ of bounded degree Benjamini-Schramm converges to a rooted graph $(Y,y)$, if for every $R>0$, the probability of the balls of radius $R$ along the sequence of graphs being graph isomorphic to the graph $(Y,y)$ is equal to one. More precisely, it satisfies 
	 	
	 	$$ \lim\limits_{i\to \infty} \dfrac{|\{x \in X_i; B_R^{X_i}(x) \cong (Y,y)|\}}{| V(X_i)|} =1 $$
	 	
	 \end{defi}
	
	Notice that a box space of finitely generated residually finite group, considered as a sequence of graphs, always Benjamini-Schramm converges to the respective Cayley graph. 

	Let us now fix a finitely generated sofic group $\Gamma$ with a sofic approximation $\mathcal{G}=\{X_i\}_{i\in\N}$. We denote by $(X,d)$ the respective space of graphs. We recall the following results from the literature.

   \begin{thm}\cite[Theorem 1.1]{alekseev2016sofic}\label{VadimTheo}
		Let $\Gamma$ be a finitely generated sofic group with a sofic approximation $\mathcal{G}$.
		\begin{enumerate}[$(1)$]
			\item If the sofic approximation $\mathcal{G}$ has property A then $\Gamma$ is amenable.
			\item If the sofic approximation $\mathcal{G}$ is asymptotically coarsely embeddable into a Hilbert space then $\Gamma$ is a-T-menable.
		\end{enumerate}
	\end{thm}
	
	\begin{prop}\cite[Theorem 0.7]{kaiser2019combinatorial} The sofic approximation
		$\mathcal{G}$ is hyperfinite if and only if the group $\Gamma$ is amenable.
	\end{prop}
	
	Even more, in the same theorem, Kaiser proved that it is also equivalent to the sequence of graphs have property almost-A, that is a measurable version of property A (see Definition \ref{propertyalmostAdef}). In this work we systematicalle study measurable versions of property A and asymptotic coarse embeddability by relating them to measure-theoretic properties of the boundary groupoid.

	\section{Coarse groupoids and their properties}\label{chapGroupoid}
	
	The aim of this section is to present the necessary definitions and results related to the coarse groupoid and one reduction of it (coarse boundary groupoid). Some notion of coarse geometry will be used here. For an introduction to coarse geometry we recommend \cite{roe2003lectures} or \cite{bunn2011bounded}. 
	
	The coarse boundary groupoid will be fundamental for our main results of this work. The coarse boundary groupoid of a space of graphs can be related to the ultraproduct of the sequence of graphs, as we will see on the next subsection. This fact help us to define a measure on this groupoid. The coarse boundary groupoid is a topological object, but if we add a measure on its space of objects, the coarse boundary groupoid can be seen as a measurable groupoid using the fact that it has a ultraproduct point of view.
	
	For this reason we introduce here the set up for the measurable case as well. Even more, we extend the group properties, like amenability and a-T-menability, to the groupoid, in both situations: topologically and measurably. The ultraproduct set up will appear only on the next section.
	
	\subsection{Groupoids and their properties}
	
\subsubsection{Topological groupoids}
	A groupoid is an algebraic structure that generalizes the notion of the group. There are different ways to define a groupoid using algebraic relations or categorical ones, for example. For more information about this object see \cite{anantharaman2001amenable}, \cite{brown2008textrm} or \cite{renault2006groupoid}. We start with some examples of topological groupoids.

	\begin{ex}
		Let $X$ be a locally compact topological space and $\Gamma$ a group acting on the space $X$. Then, $ G^{(1)} = X \times \Gamma$ is a groupoid with the space of objects $X$ and maps give by 
		$$ s(x,g)=  x g \hspace{2cm} r(x,g)= x,$$ 
		and the operation $(x,g)\cdot (xg,h)= (x,gh)$. This groupoid is called \textit{crossed product groupoid} of $X$ by $\Gamma$ and denoted by $X\rtimes \Gamma$.
	\end{ex}

	\begin{ex}
		Let $X$ be a topological locally compact space. Take $G^{(1)}= X \times X$, $ G^{(0)}= X $ and the maps $\pi_2$ and $\pi_1$ the projections on the second and first coordinates as the source and the range map, respectively. That means 
		$$ s(x,y)= \pi_2(x,y)=y $$
		$$r(x,y)=\pi_1(x,y)=x $$
		
		So, the natural choice for the multiplication is the operation $\cdot: G^{(2)} \to G^{(1)} $ given by $(x,y)\cdot (y,z)= (x,z) $.
		This is called the \textit{pair groupoid}.
	\end{ex}

	Every groupoid induces an equivalence relation on the set of objects $G^{(0)}$. Given a groupoid $G$, we define a relation $x \sim y$ in $G^{(0)}$ if and only if there is $g \in G^{(1)}$ such that $ s(g)=x, r(g)=y $, i.e, if exist some arrow connecting $x$ to $y$. By the definition of the groupoid, this is an equivalence relation. The relation set is given by
	$$R_G \colon = \{(x,y)\in G^{(0)}\times G^{(0)}; x\sim y  \} = \{(s(g),r(g))\in G^{(0)}\times G^{(0)};g \in G^{(1)} \}.$$  
	
	We can also induce a surjective groupoid homomorphism between the groupoid $G$ and the related equivalence relation $R_G$ by $$\pi: G \twoheadrightarrow R_G$$ $$\hspace{1.6cm} g\mapsto (s(g),r(g)).$$
	
	The kernel $\ker \pi = \{g \in G; s(g)=r(g)\} $ of this homomorphism is a subgroupoid. We have an isomorphism between $G$ and $R_G$, if and only if $\ker \pi$ is trivial. Moreover, there is an exact sequence $ G_{\stab} \hookrightarrow G \twoheadrightarrow R_G$, given by the inclusion, where $G_{\stab}$ is the stabilizer set $\{g\in G^{(1)}; g \cdot x = x, \forall x \in G^{(1)}\}$. Thus, $G\cong R_G / G_{\stab}$. This quotient and the isomorphism are a bit more complicated, but in the case of a groupoid with the pair operation, the stabilizer is trivial, that means, $G\cong R_G$. For the details of this construction in the general case, see \cite[section 1]{goehle2009groupoid}.

	\subsubsection{Measured groupoids}
	
	In this section, we introduce the basic terms of a measurable groupoid, one can see more about it on \cite{sauer2002l2} or \cite{anantharaman2001amenable}. As we commented in the introduction of this section, we can attach a measure on the coarse boundary groupoid, and in this case, the groupoid will be a measurable groupoid. We expect that the reader knows a little about Measure Theory. If not, a good introduction to  this area can be found in \cite{roberts2001measure}.
	
	\begin{defi}
		A discrete \textit{measurable groupoid} $G$ is a groupoid with a measure space structure on $G^{(1)}$ such that the composition map on $G^{(1)} \times G^{(1)}$ and the inverse map on $G^{(1)}$ are measurable maps and $s^{-1}(x)$ is countable, for all $x\in G^{(0)}$ (or, $r^{-1}(x)$ is countable). In this case, the source map and the range map are measurable maps and $G^{(0)}$ is a measurable subset of $G^{(1)}$ via the inclusion map. \end{defi}

	If we have a measurable groupoid $G$ with a measure $\mu$ on the space of objects $G^{(0)}$, we can define the \textit{left counting measure} $\mu_s$ induced by $\mu$ in $G^{(1)}$ on all measurable sets $A\subseteq G^{(1)}$ as
	$$ \mu_s(A) \coloneqq \int_{G^{(0)}}  |s^{-1}(x)\cap A|  \hspace{0.2cm}d\mu(x).$$
	
	The same holds for the \textit{right counting measure} $\mu_r$ of $\mu$ switching the source map with the range map. We call a measure on $G^{(0)}$ \textit{invariant} if $\mu_s=\mu_r$. Since this measure is now defined in $G^{(1)}$, this lift is denoted by $\hat{\mu}$.

	\begin{defi}\label{measured groupoid}
		A \textit{measured groupoid} $(G,\hat{\mu})$ is a measurable groupoid with an invariant measure $\mu$ on $G^{(0)}$, i.e, $\mu_s=\mu_r=\mu$, where $\hat{\mu}$ is the lift of $\mu$ to the set of morphisms.
	\end{defi}
	
	One can find in \cite{anantharaman2001amenable} a more general theory for measured groupoids. For them, a \textit{measured groupoid} is a triple $(G,\lambda^x,\mu)$, where $G$ is a measurable groupoid, $\{\lambda^x\}$ is the Haar system (i.e, a family of probability measures $\{\lambda^x\}_{x \in G^{(0)}}$ such that the support of $\lambda^x$ is contained in the range fiber $G^x=\{g\in G^{(1)}; r^{-1} (x)=g\}$) and $\mu$ is an invariant measure. But since we will work with discrete measurable groupoids that have a family of counting measures as the Haar system, we omitted the family $\{\lambda^x\}_{x \in G^{(0)}}$ of probability measures.
	
	\subsubsection{Amenable and a-T-menable groupoids}

	We introduced amenability and a-T-menability for discrete groups on the first section. In this section, we extend those properties for groupoids. Unlikely for groups, we have a topological and measurable version of those properties.

	\begin{defi}
		A discrete groupoid $G$ with compact space of objects is \emph{topologically amenable} if there is a sequence of continuous functions $\phi_n: G \to \R$ with compact support such that:
		\begin{enumerate}[(1)]
			\item for all $n \in \N$ and $y\in G^{(0)}$, we have $$\sum_{x\in G^y} \phi_n(x)=1 ;$$
			\item the sequence $$\sum_{y\in G^{s(x)}} |\phi_n(xy)-\phi_n(y)|$$ tends to zero uniformly for all $x$ in compact set of $G$.
		\end{enumerate}
	\end{defi}

	\begin{rem}
		One can define a topologically amenable groupoid as a groupoid with an invariant mean, as we did for groups.
	\end{rem}
	
	We take the definition of measurable amenability of groupoids, known as Weak Reiter's Condition, described in \cite{anantharaman2001amenable} to compare with the above topological version.
	
	\begin{defi} \label{AmeGroupoid}
		Let  $(G,\hat{\mu})$ be a measured groupoid with the measure $\mu$ in $G^{(0)}$. We say that the groupoud $G$ is \textit{measurably amenable} if there is sequence $(\phi_n)_{n\in\N}$ of functions in $L^\infty(G^{(0)},\ell^1(G,\lambda^x))$ such that
		
		\begin{enumerate}[(1)]
			\item for all $x \in G^{(0)}$ and $n\in \N$, each function $\phi_n$ is normalized, i.e, $$ \sum_{x \in G^y}(\phi_n(x))  =1;$$
			\item for all $f\in L^1(G)$ and $x\in G^{(1)}$, 
			$$\lim_{n\to \infty}  \sum_{y\in G^{x}} \sum_{z\in G^{s(y)}} f(x) |\phi_n(yz)-\phi_n(z)|=0.$$
		\end{enumerate} 	
	\end{defi}

	We have some trivial examples of amenable groupoids coming from well-known examples of amenable groups.  In the case where the group $\Gamma$ is amenable then the crossed product groupoid $X\rtimes \Gamma$ is amenable. The converse is not always true but if $X\rtimes \Gamma$ is amenable and $X$ has a $\Gamma-$invariant measure then $\Gamma$ is amenable as noticed in the Example 2.7.(2) of \cite{anantharaman2002amenability}.
	
	Now we can focus on a-T-menability. We saw the similar definition for groups but we write it here for groupoids to make clear where the elements and maps are situated. For a general view to this property on groupoids we recommend \cite{anantharaman2013haagerup}.

	\begin{defi}
		Let $G$ be a groupoid. A map $\psi:G \to \R$ is a \emph{conditionally negative definite function} if it is:
		\begin{enumerate}[$(i)$]
			\item \emph{Normalized on the diagonal}, i.e, $\psi (x,x)=0$, for all $x\in G^{(0)}$;
			\item \emph{Symmetric}, i.e, for all $y \in G^{(1)}$, we have  $\psi(y)=\psi(y^{-1});$
			\item \emph{Conditionally negative type}, i.e, for all $x\in G^{(0)}$, given any points $y_1,\dots, y_n \in G^x$ on the range fiber of $x$ and for any real numbers $\lambda_1,\dots,\lambda_n$ such that $\sum\limits^n_{i=1}\lambda_i = 0$, $\psi$ satisfies $$\sum_{i,j=1}^n  \lambda_i \lambda_j \psi(y_i^{-1}y_j) \leq 0 .$$ 
		\end{enumerate}
	\end{defi}
	
	%
	
	\begin{defi}\label{aTmean}
		
		\begin{enumerate}[$(1)$]
			\item \label{aTmeanTOP} A topological locally compact Hausdorff groupoid $G$ is a-T-menable if there is a proper continuous conditionally negative definite function $\psi: G\to \R$. Recall that \emph{proper} means that the preimage of compact sets is compact. 
			
			\item \label{aTmenaMEA} A measured groupoid $(G,\hat{\mu})$ is a-T-menable if there is a measurably proper conditionally negative definite function $\psi: G\to \R$. Now, the map $\psi $ is a \emph{measurably proper} if, for all $C>0$, the measure $\hat{\mu}(\{g\in G; \psi(g)\leq C\})$ is finite.
		\end{enumerate}
	\end{defi}
	
	Like for groups, given a topological groupoid that attached with a measure is also a measured groupoid, we have that if this groupoid is topologically a-T-menable then it is also measurably a-T-menable and the same is valid for amenability. 
	
	Similarly with the amenable example, if we have a discrete a-T-menable group $\Gamma$ acting in a measurable space $(X,\mu)$, then the crossed product groupoid $X\rtimes\Gamma$ is a-T-menable. Indeed, if there is a conditionally negative definite function $\psi: \Gamma\to \R$. Thus, we can extend it to $\psi': X\rtimes\Gamma\to \R$
	given by $\psi'(x,g)=\psi(g)$  that is also a proper conditionally negative definite function on the groupoid. The converse is not always true, but if $X\rtimes\Gamma$ is a measurably a-T-menable groupoid such that the action from $\Gamma$ on $X$ is free and preserves the measure then $\Gamma$ is an a-T-menable group.
	
	\subsection{Coarse groupoids}
	
	In this section we finally define the coarse groupoid and its reduction, the coarse boundary groupoid. For both objects, we need some coarse geometry properties to realize that the coarse groupoid carries coarse properties of the original space, as noticed in \cite{skandalis2002coarse}.
	
\subsubsection{The coarse groupoid}

	Let $X$ be an uniformly discrete metric space of bounded geometry. Consider the collection of sets $$E_R=\{ (x,y) \in X\times X; d(x,y)\leq R\},$$ for every $R>0$. Those $E_R$'s are also known by the name of \textit{entourages}. We take $\mathcal{E}$ to be the coarse structure generated by this collection. This coarse structure $\mathcal{E}$ is called the \textit{metric coarse structure}.
	
	Define 
	
	$$G(X) = \bigcup_{R>0} \overline{E_R}^{\beta X \times \beta X},$$
	where the closure of $E_R$ is the Stone-Čech compactification in $\beta X \times \beta X$. This compactification is first inside $\beta(X\times X)$, since $E_R\subseteq X\times X$, but we can extend the inclusion $E\to X\times X$ to an injective homeomorphism $\overline{E_R} \to \beta X \times \beta X$ via the universal property of the Stone-Čech compactification. This is essential to define a partial operation on $\beta X \times \beta X$. For the details on this construction see \cite[section 10.3]{roe2003lectures}.

	\begin{defi}
		Let $G(X)$ be the set defined above. The \emph{coarse groupoid} is given by $G^{(1)}=G(X)$ with the pair groupoid operation and the projections $\pi_2$ and $\pi_1$ as the respective source and range maps over $G^{(0)}= \beta X$. We will keep denoting it by $G(X)$. 
	\end{defi} 
	
	By results of \cite{skandalis2002coarse}, the coarse groupoid $G(X)$ is an locally compact Hausdorff étale topological groupoid. 

	The coarse groupoid captures the coarse properties of the related metric space in the following sense:
	
	\begin{thm}[{\cite[Theorem 5.3 and Theorem 5.4]{skandalis2002coarse}}]\label{SkandalisTuYu} 
		Let $X$ be an uniformly discrete space with bounded geometry.\begin{enumerate}[$(1)$]
			\item The space $X$ has property A if and only if the topological coarse groupoid $G(X)$ is topologically amenable.
			\item The space $X$ admits a coarse embedding into a Hilbert space if and only if the topological coarse groupoid $G(X)$ is topologically a-T-menable.
		\end{enumerate}
	\end{thm}
	

	\subsubsection{The coarse boundary groupoid}

	On the construction of the coarse groupoid we start with a uniformly discrete metric space of bounded geometry $X$. Notice that $X$ is an open subset of the compactification $\beta X$. Then, $X$ is a saturated subset of $G^{(0)}$, i.e, for every $g \in G^{(1)}= \beta X \times \beta X$ with $s(g)\in X$, we have $r(g)\in  X$. So, we can restrict the coarse groupoid, taking $G^{(0)}= X$ and  $$G^{(2)}\big|_X=\{(g,g')\in G^{(2)}; s(g), r(g)=s(g'), r(g') \in X\}.$$ 
	Naturally, we can also reduce this groupoid doing the restriction to the Stone-Čech boundary $\partial\beta X$, as $\beta X = X \sqcup \partial \beta X$.
	
	\begin{defi}
		Given a uniformly discrete metric space of bounded geometry $X$. We denote the reduction of the coarse groupoid $G(X)|_{\partial \beta X}$ associated to $X$ by $\partial G(X)$ and call it \textit{coarse boundary groupoid}.
	\end{defi}

	\begin{ex}
		According to \cite[Proposition 2.5]{finn2014spaces}, if a finitely generated discrete group $\Gamma$ acts on a uniformly discrete metric space of bounded geometry $X$ such that the induced action on $\beta X$ is free on $\partial \beta X$ and the action generates the metric coarse structure at infinity, (i.e, if for every $R>0$, there is $S_1,\cdots,S_n\in \mathcal{E} = \{E_g\}_{g\in \Gamma}= \{(x,xg); x\in X\}_{g\in \Gamma}$ and a finite subset $F\subset X\times X$ such that $E_R\subseteq (\cup_{i=1}^n S_n) \cup F$), then $\partial G(X)\cong \partial \beta X \rtimes \Gamma$. If we take a finitely generated residually finite group $\Gamma$, we know that $\Gamma$ acts on its box space $\square_{\Lambda_i} \Gamma$. Thus, $\partial G(\square_{\Lambda_i}\Gamma)\cong \partial\beta( \square_{\Lambda_i}\Gamma) \rtimes \Gamma$. 
	\end{ex}
	
	

We note some results from the literature relating the coarse boundary groupoid to the geometry of the space of graphs. 
\begin{thm}[{\cite[Section 4]{willett2006some},\cite[Lemma 5.3]{willett2015random}}]\label{Rufus}
\begin{enumerate}[$(1)$]
\item A space of bounded geometry $X$ has property A if and only if the course boundary groupoid $\partial G(X)$ is topologically amenable.
\item Let $\mathcal{X}=\{X_i\}_{i\in \N}$ be a sequence of finite graphs with bounded degree that admits an asymptotic coarse embedding into a Hilbert space and $X$ the related space of graphs, then $\partial G(X)$ is topologically a-T-menable.
\end{enumerate}
\end{thm}

 It is natural to ask whether asymptotic coarse embeddability of $X$ is equivalent to a-T-menability of the coarse boundary groupoid $\partial G(X)$, and this question was indeed raised by Willett in \cite{willett2015random}. We will provide an answer to this in Theorem \ref{thmtopaTmenable}.

	\subsection{Measure theory on the coarse boundary groupoid}
	
	The aim of this section is to exploit the existence of a measure on the coarse boundary groupoid along a non-principal ultrafilter, in the special case that $X$ is the space of graphs of a sequence of finite bounded degree graphs. The measure comes from the limit of all counting measures on the sequence of finite graphs.

	We construct here the ultraproduct in the particular case of the space of graphs coming from a sequence of finite bounded degree graphs. It is easily extendable to the general metric space case or even for normed spaces. For more about ultraproducts of graphs, see \cite{conley2013ultraproducts}.

	Given a sequence $\mathcal{X}=\{X_i\}_{i \in \N}$ of finite graphs with bounded degree such that the cardinality of each graph $ X_i$ is going to infinity when $i$ goes to infinity. Fix a non-principal ultrafilter $\omega\in \partial\beta \N$.
	
	\begin{defi}
		The \emph{ultraproduct} of the sequence of graphs $(X_i)_{i\in\N}$ is the set $$\prod_{i\to\omega} X_i \coloneqq\bigg(\prod_{i\in\N} X_i\bigg) \bigg/  \sim_\omega, \text{ where } (z_i)_{i\in \N} \sim_\omega (y_i)_{i\in \N} \text{ iff }\{i\in \N; z_i=y_i\} \in \omega.$$
	\end{defi} 
	
	\begin{rem}
		The graphs are also metric spaces so we can write the equivalence relation as $$ (z_i)_{i\in \N} \sim_\omega (y_i)_{i\in \N} \text{  if  and only if }\{i\in \N; d_i(z_i,y_i)=0\} \in \omega $$
		Even more, the sequence of length distances $(d_i)_{i\in \N} $ on each graph $X_i$ induces a metric on the ultraproduct of the sequence of graphs given by $ d_\omega((y_i)_i,(z_i)_i)= \lim\limits_{i \to \omega} d_i(y_i,z_i)$ along the non-principal ultrafilter $\omega \in \partial \beta \N$. Notice that the metric $d_\omega$ might be not finite, so we allow the metric take values on $[0,\infty]$.
	\end{rem}

	When necessary, we denote by $[y_i]_\omega$ or $\lim\limits_{i\to\omega} y_i$, an element $y \in\prod\limits_{i\to\omega} X_i$. Even more, sometimes we use the notation $\prod\limits_{i\to\omega}(X_i,d_i) $ for ultraproduct, to make clear the distance in each graph $(X_i,d_i)$. We will use this notation for sofic groups and in the more general situation for measured spaces.

	In other to guarantee the convergence of the sequences some authors attach a fixed sequence $x_i \in X_i$ along the sequence of graphs and define
	$$\prod\limits_{i\in \N} \big(X_i,x_i,d_i\big)=\bigg\{(y_i)_{i\in\N} ; \sup_{i\in \N} d_i(x_i,y_i) < \infty \bigg\} \hspace{0.3cm} \text{and}  \hspace{0.3cm} d_\omega((y_i)_i,(z_i)_i)= \lim\limits_{i \to \omega} d_i(y_i,z_i).$$

	Notice that $\prod\limits_{i\in \N} \big(X_i,x_i\big)$ is the set of sequences finitely close to $(x_i)_{i\in\N}$. Thus, the distance $d_\omega$ is well-defined and $ \big( \prod\limits_{i\in \N} \big(X_i,x_i\big) ,d_\omega \big)$ is a metric space by the continuity of the limit. Putting in the same equivalence class, the sequences that converge to the same point along the non-principal ultrafilter $\omega \in \partial \beta \N$, i.e, $(y_i)_i \sim_\omega (z_i)_i $  if and only if $d_\omega((y_i)_i,(z_i)_i)=0$. We denote by $\prod\limits_{i\to \omega} \big(X_i,x_i,d_i\big)$ the quotient of $\prod\limits_{i\in \N} \big(X_i,x_i,d_i\big)$ by this equivalence relation.
	
	One can do the ultraproduct of the pair of graphs $X_i \times X_i$ with the sequence of metric $d_i$, where $\{X_i\}_{i\in\N}$ is a sequence of graphs. More specifically, if we take the ultraproduct $\prod\limits_{i\to\omega} (\{(x_i,y_i)\in X_i \times X_i; d_i(x_i,y_i)\leq \infty\},d_i)$, it give us a description of the coarse boundary groupoid of the space of graphs as notice in \cite[Example 2.33]{carderi2018non}. We will use this fact later on the text when we talk about measured ultraproducts. In the same way we did for the Stone-Čech compatification of $\N$, one can see a point in $\partial \beta X$ as a limit of a sequence along a non-principal ultrafilter $\omega \in \partial \beta \N$. Based on that, an important geometric relation between them was proved in \cite{alekseev2016sofic} and written here for completeness:

	\begin{prop}[{\cite[Proposition 2.14]{alekseev2016sofic}}]\label{isometric}
		Let $\partial G(X)$ be the boundary groupoid, where $X$ is a space of finite bounded degree graphs. For a fixed non-principal ultrafilter $\omega \in \partial \beta \N$ and a point $\eta\in \partial \beta X$ such that $\eta =  [x_i]_\omega $, the source fiber $\big(\partial G( X)\big)_\eta =\{(\eta_1,\eta)\in \partial\beta X \times \partial\beta X \} $ with the metric $ d_\eta\big((\eta_1,\eta), (\eta_2,\eta) \big) = \inf \{R>0 ; (\eta_1,\eta_2)\in \overline{E_R}  \}$ has base point isometry with $\bigg(\prod\limits_{i\to\omega} X_i,x_i,d_i\bigg)$ via $$f:\bigg(\prod\limits_{i\to\omega} X_i,x_i,d_i\bigg) \to \big(\partial G(X)\big)_\eta$$  $$\hspace{1.6cm}\big([y_i]_\omega\big) \mapsto ([y_i]_\omega,\eta).$$
	\end{prop}

	Notice that, given $x\in\partial\beta X$, there are a non-principal ultrafilter $\omega \in \partial\beta \N$ and sequence $(x_i)_{i\in\N}$ such that $x=\lim\limits_{i\to\omega} x_i=[x_i]_\omega$. This approximation can generate many different choices, but the proposition above proves that those choices have isometric source fibers. This result assists us to define a measure on $\partial\beta X$ induced by the ultralimit of the counting measure $\mu_i$ in each finite graph $X_i$, i.e, $\mu_i(A)= \frac{|A|}{|X_i|}$, for all subsets $A\in X_i$. 
	
	Fix a non-principal ultrafilter $\omega \in \partial\beta\N$, we can obtain a measure $\mu_\omega$ on $G^{(0)}=\partial \beta X$ related to the state $\tau_\omega: C(\beta X)\to \R$ defined by 
	$$ \tau_\omega(f) = \lim\limits_{i \to\omega} \dfrac{1}{|X_i|} \sum_{x \in X_i} f(x).$$
	Our measure of sets on a clopen subset  $A \subseteq \partial  \beta X$ is given by $$\mu_\omega(A)= \tau_\omega(\chi_{[A_i]_\omega}) =\lim\limits_{i \to\omega} \dfrac{1}{|X_i|} \sum\limits_{x \in X_i} \chi_{A_i}(x)= \lim\limits_{i\to\omega} \mu_i (A_i),$$ where $\chi$ denotes the respective characteristic function. Note that $\mu_\omega(X)=0$ and $\mu_\omega(\partial \beta X)=1$, moreover, $\mu_\omega$ is a probability measure over a non-standard space $\partial \beta X$. With this measure on $\partial \beta X$ we can lift this measure to the coarse boundary groupoid and then have a measured groupoid as we saw in the last section. 
	
	As we have seen before, the coarse boundary groupoid $\partial G(X)$
	is a topological groupoid. For a given non-principal ultrafilter $\omega \in \partial \beta \N$, we can attach with it a measure $\mu_\omega$ on the base space $\partial \beta X$. In this situation, $(\partial G(X), \hat{\mu}_\omega)$ is a measured groupoid, where $\hat{\mu}_\omega$ is the lift of $\mu_\omega$ on $\partial G(X)$.
	Sometimes we refer to the measured coarse boundary groupoid $(\partial G(X), \hat{\mu}_\omega)$ by it equivalence relation $(R_{\partial G(X)},\hat{\mu}_\omega)$ when we talk about hyperfiniteness. We are allowed to do this because of the isomorphism between these objects and it is more convenient to keep the notation common in ergodic theory, when we refer to hyperfiniteness of measured spaces.

	\subsubsection{Measured Ultraproducts}

	Another way to construct the measure $\mu_\omega$ on $\partial\beta X$ is by using Caratheodory's Theorem and the theory of ultraproducts of measure spaces; we refer to \cite[Section 1.1.3]{carderi2015groups}.
	
	Take a sequence of measurable spaces $\{(X_i,\mathcal{B}_i,\mu_i)\}_{i\in \N}$, where $\mu_i$ the respective measure over the $\sigma$-algebra $\mathcal{B}_i$. Let $\prod\limits_{i\to\omega} X_i$ be the ultraproduct of $\{X_i\}_{i\in\N}$ for a fixed non-principal ultrafilter $\omega \in \partial \beta \N$. Define a map $\theta_\omega$ from the power set of the ultraproduct $P\bigg(\prod\limits_{i\to\omega} X_i\bigg) $ to $[0,\infty)$, for $A \in P\bigg(\prod\limits_{i\to\omega} X_i\bigg)$, as
	$$ \theta_\omega(A)=\inf\bigg\{\sum_{n\in \N}\lim_{i \to \omega} \mu_i(B_i^n); A \subseteq \bigcup_n [B_i^n]_\omega, B_i^n \in \mathcal{B}_i \bigg\}.$$
	
	This map defined above is an outer measure, i.e, a monotone, countably subadditive map with $\theta_\omega(\emptyset)=0$. The \textit{measured ultraproduct} along a non-principal ultrafilter $\omega \in \partial \beta \N$ of a sequence of measure spaces $\{(X_i,\mathcal{B}_i,\mu_i)\}_{i\in \N}$ is the triple $\bigg(\prod\limits_{i\to\omega} X_i, \mathcal{B}, \theta_\omega\bigg)$, where $\bigg(\prod\limits_{i\to\omega} X_i\bigg)$ is the ultraproduct as we saw before with the measure $\theta_\omega$ over the $\sigma-$algebra $$\mathcal{B}=\bigg\{A \subseteq \prod\limits_{i\to\omega} X_i;\hspace{0.1cm} \theta_\omega(C) \geq \theta_\omega(C\cap A) + \theta_\omega(C \cap A^c), \hspace{0.3cm} \forall C \subseteq \prod\limits_{i\to\omega} X_i \bigg\}.$$
	
	By Caratheodory's Extension Theorem,  $\bigg(\prod\limits_{i\to\omega} X_i, \mathcal{B}, \theta_\omega\bigg)$ is a measured space (see \cite{roberts2001measure}). Even more with the next result one can compute the measure along the approximation using the ultraproducts. A more general version of the next lemma is proved in \cite{carderi2015groups}.  
	
	\begin{lemma}[{\cite[Proposition 1.1.7]{carderi2015groups}}]\label{subsetUltra}
		Let  $\bigg(\prod\limits_{i\to\omega} X_i, \mathcal{B}, \theta_\omega\bigg)$ be the measured ultraproduct of the sequence of measured spaces $\{(X_i,\mathcal{B}_i,\mu_i)\}_{i\in \N}$, then; 
		\begin{enumerate}[(1)]
			\item For all sequences $\{ B_i\}_{i\in \N} \in \mathcal{B}_i$, the class $[B_i]_\omega$ is in $\mathcal{B} $ and $\theta_\omega( [B_i]_\omega ) = \lim\limits_{i \to \omega} \mu_i(B_i)$;
			\item For all $A\in \mathcal{B}$, there is a sequence $\{ B_i\}_{i\in \N} \in \mathcal{B}_i$ such that $\theta_\omega(A \triangle [B_i]_\omega)=0.$
			
		\end{enumerate}
	\end{lemma}

	In the particular case of the space of graphs, the measured ultraproduct of the sequence $\{(X_i,\mathcal{B}_i,\mu_i)\}_{i\in \N}$, where $\{X_i\}_{i \in \N} $ is the sequence of bounded degree finite graphs, $\mathcal{B}_i$ is the $\sigma-$algebra of the metric space and $\mu_i$ the counting measure in each graph $X_i,$ is  $(\partial\beta X, \mathcal{B}, \mu_\omega)$. Notice that the measure defined beforehand coincide, that is, $\theta_\omega=\lim\limits_{i \to \omega}\mu_i=\mu_\omega$. Not only the lift of the measures from the graphs to the groupoid is possible. The structure of the ultraproduct allows us to lift measurable bounded functions defined on it.

	\subsection{Approximation of functions defined on the coarse boundary groupoid}

	As we saw in the last section, when $\partial G(X)$ is the coarse boundary groupoid of a space of finite bounded degree graphs $X$, we can attach it with the measure $\hat{\mu}_\omega$ such that the coarse boundary groupoid is a measured groupoid. From now on, all claims will be restricted for this case. The goal of this section is to describe how one can lift functions from the measurable coarse boundary groupoid back to their space of graphs.

	As the groupoid $(\partial G(X),\hat{\mu}_\omega)$ can been seen as an ultraproduct of $X \times X$ where the distance of points are finite. We can work with approximations it that are compatible metric and measurable wise. For example, given an non-principal ultrafilter $\omega \in \partial \beta \N$ and $R >0$, the closure of the entourages $\overline{E_R}=\{(a,b)\in  \partial\beta X \times \partial\beta X ; d(a,b)\leq R\}\subseteq \partial G(X)$ can be described as a measurable limit of the sets $E_R^i= \{(a_i,b_i)\in X_i \times X_i ; d(a_i,b_i)\leq R\}$. More precisely, from the last Lemma, one can see that $\hat{\mu}_\omega([E_R^i]_\omega \triangle \overline{E_R})=0$.For this reason, sometimes we saw that $[E_R^i]_\omega = \overline{E_R}$ almost everywhere. Moreover, functions defined on the measurable coarse boundary groupoid also admits an approximation.
	
	\begin{defi}
		Fix an non-principal ultrafilter $\omega \in\partial \beta \N$. We say that a function $\psi \in L^\infty({\partial G(X)}, \hat{\mu}_\omega)$ \emph{admits an approximation}, if there is a sequence of functions $\psi_i \in L^\infty(X_i \times X_i, \hat{\mu}_i)$ such that, for every $k\in \N$, $\psi|_{\overline{E_k}}= [\psi_i|_{E_k^{i}}]_\omega$ almost everywhere, where $E_k^i$ is the sequence of sets $\{(x_i,y_i) \in X_i \times X_i; d(x_i,y_i)\leq k\}$ that describes almost everywhere the entourage $\overline{E_k}= [E_k^i]_\omega=\{(x,y) \in \partial\beta X\times \partial \beta X; d_\omega(x,y)\leq k\}$.	
	\end{defi}
	
	The next results can also be founded in the Remark 3.12 of \cite{alekseev2016sofic} with an operator algebra approach and in the case where $X$ is the space of graphs of a sofic approximation. Here we claim for a more general situation.

	\begin{lemma}\label{appfunc}
		All functions $\psi \in L^\infty({\partial G(X)}, \hat{\mu}_\omega)$ admit an approximation.
	\end{lemma} 
	
	\begin{proof}
		We can describe the coarse boundary groupoid $({\partial G(X)}, \hat{\mu}_\omega) $ by the ultraproduct $\prod\limits_{i\to\omega} (\{(x_i,y_i)\in X_i \times X_i; d_i(x_i,y_i)\leq \infty\},d_i,\mu_i)$. Even more, it is almost everywhere seen as $ \bigcup_{k>0} \prod\limits_{i\to \omega} (E_k^i,\hat{\mu}_i)$. Define for $i,k \in \N$, the set $S_{k+1}^i\coloneqq E_{k+1}^i \setminus E_k^i$. Notice that $S_{k}^i$ is disjoint from $S_{k'}^i$ for all $k\neq k'$. So, we have the following approximation
		$$\bigsqcup_{j=0}^{k+1} \prod\limits_{i\to \omega} (S_j^i,\hat{\mu}_i)= \prod\limits_{i\to \omega} (E_k^i,\hat{\mu}_i).$$ 
		
		Thus,  $({\partial G(X)}, \hat{\mu}_\omega) $ can be written almost everywhere as $ \bigsqcup\limits_{k>0}\prod\limits_{i\to \omega } (S_k^i,\hat{\mu}_i)$. 
		Now take $\psi\in L^\infty({\partial G(X)},\hat{\mu}_\omega) \cong L^\infty\bigg( \bigsqcup\limits_{k>0}\prod\limits_{ i\to\omega } S_k^i,\hat{\mu}_i\bigg)\cong  \bigsqcup\limits_{k>0}\prod\limits_{ i\to\omega } L^\infty\bigg( S_k^i,\hat{\mu}_i\bigg) $, we can write $\psi= \sum\limits_{k=0}^{\infty} \psi^k$, where
		$$ \psi^k =\psi|_{\prod\limits_{ i\to \omega } (S_k^i,\hat{\mu}_i)} = [(\psi_i^k)]_\omega $$

		Each $ \psi_i^k $ is supported in $S_{k}^i$, thus, given $(x_i,y_i)\in X_i$ such that $d(x_i,y_i)=k$, we take $\psi_i:X_i\times X_i \to \R$ defined by $\psi_i(x_i,y_i)\coloneqq \psi_i^k(x_i,y_i)$, that satisfies the
		required properties. \end{proof}


	For unbounded functions, we ask for more specific notion of approximation, but still fitting with the one defined before.

	\begin{defi}
		Fix a non-principal ultrafilter $\omega \in\partial \beta \N$. A map $\psi \in L^\infty({\partial G(X)}, \hat{\mu}_\omega)$ is \emph{represented} by a sequence of functions $(\psi_i: X_i \times X_i \to \R)_{i\in \N}$, if, for all $k\in\N$ and $L\in \Z$,  $$\psi\big|_{H_{\omega,k,L}}=\big[\psi_i \big|_{H_{i,k,L}}\big]_\omega,$$ where $H_{i,k,L}= \{ (x_i,y_i) \in X_i \times X_i;  d(x_i,y_i)=k, \psi_i(x_i,y_i)\in [L,L+1) \} $ and $H_{\omega,k,L} = [H_{i,k,L}]_\omega$.
	\end{defi}

	Abusing notation we also say that $(\psi_i)_{i\in \N}$ approximates $\psi $ if the above definition is satisfied. Notice that we used the integers only to define a nice countable partition, one can take different countable partitions of the coarse boundary groupoid if necessary.

	\begin{prop}\label{measufunc}
		For every measurable function $\psi: \big(\partial G (X),\hat{\mu}_\omega\big) \to \R$, there exists a sequence of functions $(\psi_i)_{i\in \N}$ from $X_i\times X_i $ to the real numbers that represents $\psi$.
	\end{prop}
	\begin{proof}
		Given  $L\in \Z$, define the cutoff of  $\psi$ on the set $[L,L+1)$ as $$\psi_{[L,L+1)}(\gamma)\coloneqq \bigg\{ \begin{array}{cc}
			\psi(\gamma), & \text{if } \psi(\gamma) \in [L,L+1)\\
			0, & otherwise \end{array}  $$
		and denote by $A_{\omega,L}$ the support of $\psi_{[L,L+1)}$. Since $\psi_{[L,L+1)}$ is in $L^\infty(\partial G(X),\hat{\mu}_\omega) $ and by the Lemma \ref{appfunc}, there is a sequence $(\psi_{i,L})_{i\in \N}$ that approximates  $\psi_{[L,L+1]}$ such that for every $k\in \N$, $\psi_{[L,L+1)}|_{E_k}= [\psi_{i,L}|_{E_k^{i}}]_\omega$. Without loss of generality, we can assume, for all $L\in\Z$, that $A_{i,L}\coloneqq \supp (\psi_{i,L})$ is disjoint from $A_{i,L'} \coloneqq \supp (\psi_{i,L'})$, for all $L\neq L'$ and every $i\in \N$. Indeed, since $\psi_{[L,L+1)}$ and $\psi_{[L',L'+1)}$ are bounded, so there are sequences $(\psi_{i,L})_{i\in \N}$ and $(\psi_{i,L'})_{i\in \N}$ that approximates $\psi_{[L,L+1)}$ and $\psi_{[L',L'+1)}$ respectively. Clearly, $\supp\big(\psi_{[L,L+1)}\big)$ and $\supp\big(\psi_{[L',L'+1)}\big)$ are disjoint, so $$\lim_{i\to\omega} \hat{\mu}_i \big( \supp (\psi_{i,L}) \cap \supp (\psi_{i,L'}) \big)=0.$$ Sending to zero the values of $\psi_{i,L}$, where $\supp (\psi_{i,L}) \cap \supp (\psi_{i,L'}) \neq \emptyset$, we still representing $\psi_{[L,L+1)}$. We can apply this procedure inductively for every $L\in\Z$, such that the support of each approximation is disjoint from each other.

		Take $\psi_i = \sum\limits_{L=-\infty}^\infty \psi_{i,L}$ and $\psi = \sum\limits_{L=-\infty}^\infty \psi_{[L,L+1)}$. Notice that $H_{\omega,k,L}= S_{\omega,k} \cap A_{\omega,L} $, then
		
		$$\psi\big|_{H_{\omega,k,L}} = \psi_{[L,L+1)}\big|_{S_{\omega,k}}= \bigg[\psi_{i,L}\big|_{S_k^{i}}\bigg]_\omega =\bigg[\psi_i \big|_{H_{i,k,L}}\bigg]_\omega$$ 
		
		the second equality holds because the sets $S^i_k$'s are also disjoint, for all $k\in\N$.\end{proof}

	\section{Geometry of the space of graphs from its coarse boundary groupoid}

	This section is devoted to the main results of this article; here we systematically relate topological and measurable amenability or a-T-menability of the coarse boundary groupoid $\partial G(X)$ of a space of finite bounded degree graphs $X$ to the geometry of the space of graphs.
	
	

\subsection{Amenability}
	As pointed out before, it is known that a space of graphs $X$ has property A if and only if the coarse boundary groupoid $\partial G(X)$ is topologically amenable \cite[Section 4]{willett2006some}. In this section, we consider this question in the measurable setting, ultimately proving the following
	
	\begin{thm}[see Theorem \ref{mainthm}]
		Let $\mathcal{X}=\{X_i\}_{i\in\N}$ be a sequence of finite graphs with bounded degree such that the cardinality of $X_i$ goes to infinity when $i$ goes to infinity. Take $X$ to be the space of graphs of this sequence and ${\partial G(X)}$ the related coarse boundary groupoid. The following statements are equivalent:
		\begin{enumerate}[$(1)$]
			\item The measurable coarse boundary groupoid $({\partial G(X)},\hat{\mu}_\omega)$ is measurably amenable, for every non-principal ultrafilter $\omega\in\partial\beta\N$;
			\item  The sequence of graphs $\mathcal{X}$ has property A on average along every non-principal ultrafilter $\omega\in\partial\beta\N$;
			\item The sequence of graphs $\mathcal{X}$ has property almost-A along every non-principal ultrafilter $\omega \in \partial \beta \N$;
			\item The sequence of graphs $\mathcal{X}$ is hyperfinite;
			\item The measurable equivalence relation $({\partial G(X)},\hat{\mu}_\omega)$ induced by the coarse boundary groupoid is $\mu_\omega$-hyperfinite, for every non-principal ultrafilter $\omega\in\partial\beta\N$.
		\end{enumerate}
	\end{thm}
		
Let us remark here that the main novelty of this theorem is to put the interpret the known results about hyperfiniteness and property almost-A using the language of groupoids. Indeed, as it turns out, the subtle difference between property A and hyperfiniteness turns out to be precisely the difference between topological and measured amenability of the coarse boundary groupoid.

We first collect the necessary notions.

 \subsubsection{Hyperfiniteness along an ultrafilter}

	We saw before the basic notion of hyperfiniteness for a sequence of graphs. Notice that we defined a hyperfinite in sequence of graphs along the natural numbers, but we can do it taking this property for every non-principal ultrafilter $\omega\in \partial\beta \N$.
	
	\begin{defi}
		A sequence of finite graphs $\mathcal{X} = \{X_i\}_{i \in\N}$ is  \textit{hyperfinite along a non-principal ultrafilter $\omega \in \partial \beta \N$} if, for every $\varepsilon >0$, there is a $K_{\varepsilon} \in \N$ and a partition of the vertex sets $ V(X_i) = A_1^i \cup A_2^i \cup \dots \cup A_{n_i}^i$, for all $i \in \N$, such that $|A_l^i|< K_\varepsilon$, for  $1\leq l \leq n_i$ and
		if $E_i^\varepsilon = \{(x,y)\in E(X_i); x \in A_n^i, y \in A_m^i, n\neq m\}$ then
		$$ \lim_{i \rightarrow \omega} \dfrac{|E_i^\varepsilon|}{|V(X_i)|} \leq \varepsilon .$$
	\end{defi}

	Naturally, we have the following equivalence between the definitions that is only an argument over the ultralimits:
	
	\begin{lemma}\label{hyperultra}
		Let $\mathcal{X} = \{X_i\}_{i \in\N}$ be a sequence of finite graphs. The sequence $\{X_i\}_{i \in\N}$ is hyperfinite if and only if for every non-principal ultrafilter $\omega \in \partial\beta\N$, the sequence $\{X_i\}_{i\in \N}$ is hyperfinite along the ultrafilter $\omega \in \partial \beta \N$.
	\end{lemma}
	
	\begin{proof}
		If the sequence of graphs $\{X_i\}_{i \in\N}$ satisfies the conditions of hyperfiniteness for all $i\in \N$ then, all properties hold for all $i\in \N$, since $\omega \in \partial\beta\N$ and $$\lim\limits_{i \rightarrow \omega} \frac{|E_i^\varepsilon|}{|V(X_i)|} \leq\limsup\limits_{i \rightarrow \infty} \frac{|E_i^{\varepsilon}|}{|V(X_i)|}.$$ 
		
		Now suppose the opposite, i.e, there is a $\varepsilon_0>0$ such that for all $K>0$, exist an $i_K\in \N$ that for all decompositions of $X_{i_K}$ with connected components of size smaller than $K$ satisfies $\frac{|E_{i_K}^{\varepsilon_0}|}{|V(X_{i_K})|} > \varepsilon_0 $. Note that $i_K$ tends to infinity when $K$ goes to infinity. Consider $I=\{i_K; K\in\N \} \subseteq \N$. Notice that $I$ is an infinite set, then there is a non-principal ultrafilter $\omega_{i}$ containing $I$. Thus, for all decompositions of $X_{i_k}$, we have 
		$$ \lim_{i\to\omega_{i}} \dfrac{|E_{{i_k}}^{\varepsilon_0}|}{|V(X_{{i_k}})|} > \varepsilon_0.$$
		
		That means, $\{X_i\}_{i\in \N}$	is not hyperfinite along $\omega_{i}$, contradicting with the fact that  $\{X_i\}_{i\in \N}$ is hyperfinte along all ultrafilters.	\end{proof}

	For the next result we identify the vertices of the edges that belong to $E_i^\varepsilon$ and prove that hyperfiniteness can be equivalent to removing a set with small measure from the vertex set. Recall that two edges are called \textit{incident} if they share a vertex. In the case of a sequence of finite graphs with bounded degree $d$ we have the following equivalence.
	
	\begin{prop}\label{Elekweakdef}
		Let $\mathcal{X} = \{X_i\}_{i \in\N}$ be a sequence of finite bounded degree graphs. The sequence of graphs $\{X_i\}_{i \in\N}$ is hyperfinite if and only if for every $\varepsilon>0$, there is $K_\varepsilon\in \N $ and $\{Z_i\}_{i\in\N}\subset \{X_i\}_{i\in\N}$ such that $\limsup\limits_{i \rightarrow \infty}\frac{|Z_i|}{|V(X_i)|} \leq \varepsilon$ and if we remove the edges incident to all $Z_i$'s, the resulting graph $\mathcal{X}'$ admits a decomposition  with connected components of size at most $K_\varepsilon$. \end{prop}
	
	\begin{proof}
		
		Given $\varepsilon>0$, take $\varepsilon_0 = \dfrac{\varepsilon}{2d}$. Then, there exists a $K_{\varepsilon_0}\in \N$ and a partition with the size of components bounded by $K_{\varepsilon_0}$ with
		$ \limsup\limits_{i \rightarrow \infty} \dfrac{|E_i^{\varepsilon_0}|}{|V(X_i)|} \leq \varepsilon_0 $. Let $Z_i\subset V(X_i)$ be the set of points $x \in V(X_i)$ such that exist a $y\in V(X_i)$ where $(x,y)$ is in the set $E_i^{\varepsilon_0}$. Notice that $|Z_i|\leq 2d |E_i^{\varepsilon_0}|$, where $d$ is the maximum degree of all graphs. Then, $$\limsup\limits_{i \rightarrow \infty}\frac{|Z_i|}{|V(X_i)|} \leq 2d{\varepsilon_0}=\varepsilon,$$ and by construction, $\mathcal{X}'$ has connected components of bounded size.
		
		Conversely, given an $\varepsilon>0$, there exist $K_{\frac{\varepsilon}{2}}$ and a sequence $\{Z_i\}_{i\in\N}\subset \{X_i\}_{i\in\N}$ such that $\limsup\limits_{i \rightarrow \infty}\frac{|Z_i|}{|V(X_i)|} \leq \frac{\varepsilon}{2}$ and the remaining graph $\mathcal{X}'$ has connected components of size bounded by $K_{\frac{\varepsilon}{2}}$. Then, there is a partition $V(X_i \backslash Z_i)=  A_1^i \cup A_2^i \cup \dots \cup A_{n_i}^i$
		such that $|A_l^i|< K_{\frac{\varepsilon}{2}}$, for all $i \in \N$, $1\leq l \leq n_i$. Notice that for each removed edge $(x,y)\in E_i^\varepsilon$ is the same as remove at most 2 points of $Z_i$, then 	
		
		$$ \limsup_{i \rightarrow \infty} \dfrac{|E_i^\varepsilon|}{|V(X_i)|} \leq \limsup_{i \rightarrow \infty} \dfrac{2 |Z_i|}{|V(X_i)|}\leq 2 \cdot \frac{\varepsilon}{2} = \varepsilon.$$\end{proof}

	By the previous lemma, we can also restate the last proposition in terms of ultralimits.
	
	\begin{prop}\label{Elekultrafilter}
		Let $\mathcal{X} = \{X_i\}_{i \in\N}$ be a sequence of finite bounded degree graphs. The sequence of graphs $\{X_i\}_{i \in\N}$ is hyperfinite for every non-principal ultrafilter $\omega \in \partial \beta \N$ if and only if for every $\varepsilon>0$, there are $K_\varepsilon\in \N $ and $\{Z_i\}_{i\in\N}\subseteq\{X_i\}_{i\in\N}$ such that $\lim\limits_{i \rightarrow \omega}\frac{|Z_i|}{|V(X_i)|} \leq \varepsilon$ and if we remove the edges incident to all $Z_i$'s, the resulting graph $\mathcal{X}'$ admits a decomposition  with connected components of size at most $K_\varepsilon$. 
		
	\end{prop}

\subsubsection{Hyperfiniteness of the coarse boundary groupoid}

	Hyperfinitiness in a measure equivalence relation has been studied for the last sixty years, usually restricted to standard probabilistic spaces, like in \cite{kechris2004topics}. Most of the ideas of this section are adaptations of \cite{marks2017short} but for the equivalence relation related to the measurable coarse boundary groupoid $(\partial G(X),\hat{\mu}_\omega)$ of a space of finite graphs with bounded degree $X$.
	
 The adaptation of the classical ideas requires some work because $({\partial G(X)},\hat{\mu}_\omega)$ is an equivalence relation over a non-standard probability space $\partial \beta X$. On the other hand, this equivalence relation is still countable, i.e, every equivalence class $[x]=\{y;(x,y)\in {\partial G(X)} \}$ is countable, for every $x\in \partial\beta X$. 
	
	\begin{defi}
		A countable equivalence relation $E$ on a measure space $(X,\mu)$ is \textit{hyperfinite} if exists a nested sequence $E_0\subseteq E_1\subseteq \dots $ of finite subequivalence relations such that $E=\bigcup\limits_{n\in\N} E_n$.
	\end{defi}

	As $(X,\mu)$ is a measure space, we can define the notion of hyperfinitiness almost everywhere.	A countable equivalence relation $E$ on $(X,\mu)$ is said to be $\mu-$\emph{hyperfinite} if there is a $\mu-$conull set $A \subset X$, such that $E\upharpoonright A  = E \cap (A \times A)$ is hyperfinite. In this case, we also say that $E= \cup_n E_n$ \emph{almost everywhere} or $E$ is equal to the union of $E_n$'s up to null sets.

	Note that all hyperfinite equivalence relations are $\mu-$hyperfinite. In particular, if $E$ is a hyperfinite equivalence relation on $X$, we can extract an increasing exhaustive sequence $\{E_n \}_{n \in \N}$ of finite subrelations such that $|[x]_{E_n}|\leq n $, for all $x\in X$ and $n\in \N$, that keeps $E$ as a hyperfinite equivalence relation related with this new sequence. Indeed, given an increasing exhaustive sequence of finite subequivalence relations $\{F'_n \}_{n \in \N}$  such that $E=\cup_n F'_n$. Define $k_n(x)= \max \{k \in\N ; |[x]_{F'_k}|\leq n  \} $ and now we take the subrelation:
	$$ E_0=\Delta_X = \{(x,x)\in X\times X\} $$
	$$ (x,y) \in E_n \Longleftrightarrow (x,y) \in F'_{k_n(x)} .$$
	
	It is easy to see that this is a nested increasing exhaustive sequence, and by definition, the sequence $\{E_n \}_{n \in \N}$ has the required property on the equivalence class level. Moreover, if $X$ is a metric space, for a $K\in \N$ big enough, we can extract a subequivalence relation $F\subseteq E$ with $\diam ([x]_F) \leq K$, for almost every $x\in X$. Recall that our coarse boundary groupoid has a metric coming from the graphs. Formally, in our particular case of the coarse boundary groupoid:
	
	\begin{prop}\label{mu-hyperEqui}
		The equivalence relation ${\partial G(X)}$ is $\mu_\omega$-hyperfinite if and only if for every $\varepsilon>0$ and every $S>0$, there is a $K\in\N$ and a finite subequivalence relation $F\subseteq {\partial G(X)}$ such that $$ \mu_\omega(\{x \in \partial \beta X ; \diam([x]_F)<K\})> 1- \varepsilon$$ and $$ \mu_\omega(\{x \in \partial \beta X ; \forall y \in \partial \beta X \text{ such that } d(x,y) < S \Longrightarrow (x,y)\in F \})> 1- \varepsilon
		.$$
	\end{prop} 
	
	\begin{proof}
		Assume that ${\partial G(X)}$ is a $\mu_\omega-$hyperfinite equivalence with respect to the nested sequence of finite subequivalence relations $\{E_n\}_{n \in \N}$. Given $\varepsilon>0$ and $S>0$, define, for every $n\in\N$, the set
		$$Y_n=\{x \in \partial \beta X ; \forall y \in \partial \beta X \text{ such that } d(x,y) < S \Longrightarrow  (x,y )\in E_n \}. $$
		
		As $E_n$ is a nested increasing sequence, we have $Y_n\subseteq Y_{n+1}$ in $\partial\beta X$. Since ${\partial G(X)}$ is equal to $ \cup_n E_n$ almost everywhere, for some $N \in \N$ big enough, we must have $\mu_\omega(Y_N)> 1 - \varepsilon$. So, we take $F\coloneqq E_N$ as the finite subequivalence relation.
		
		For all $x \in \partial \beta X$, take $D_x=\max\limits_{ (x,y) \in F } d(x,y)$, that is bounded by the diameter of the equivalence class $[x]_F$. We can see $D$ as a measurable map from $\partial \beta X$ to $\N$ that send every $x$ to $D_x$. Note that
		
		$$ 1=\sum_{n\in \N}  \mu_\omega(D^{-1}(n)).$$ 
		
		By the convergence of this series and for our given $\varepsilon$, there is a $K\in \N$ satisfying 
		
		$$\sum\limits_{n>K}  \mu_\omega(D^{-1}(n))< \varepsilon.$$ 
		
		Thus,
		\begin{equation*}
			\begin{split}
				1& = \mu_\omega (\partial \beta X) \\ 
				&= \mu_\omega(\{x \in \partial \beta X ; \diam([x]_F)<K\})  + \mu_\omega(\{x \in \partial \beta X ; \diam([x]_F)\geq K\})\\
				&= \mu_\omega(\{x \in \partial \beta X ; \diam([x]_F)<K\})  + \mu_\omega(\cup_{n>K} D^{-1}(n))\\
				& \leq  \mu_\omega(\{x \in \partial \beta X ; \diam([x]_F)<K\})  + \sum\limits_{n>K}\mu_\omega( D^{-1}(n))\\
				&< \mu_\omega(\{x \in \partial \beta X ; \diam([x]_F)<K\})  + \varepsilon .
			\end{split}
		\end{equation*}	
		As we wished.

		Conversely, take $\varepsilon_n = \frac{1}{2^n}$ and $S_n=n$. So, there is a sequence of finite subequivalence relations $\{F_n\}_{n \in \N} \subseteq {\partial G(X)}$ such that, for all $n\in \N$, $$ \mu(\{x \in \partial \beta X : \forall y \in \partial \beta X \text{ such that } d(x,y) < n \Longrightarrow (x,y )\in F_n \})> 1- \varepsilon_n.$$ Note that every $F_n$ has finite classes but we do not necessarily have a nested sequence. To fix that, define $E_n= \bigcap\limits_{i\geq n}F_i$. Clearly, this is a nested sequence of finite equivalence relations. Take  $$Y_n=\{x \in \partial \beta X ; \forall y \in \partial \beta X \text{ such that } d(x,y) < S_n \Longrightarrow  (x,y )\in E_n \}.$$ Noitce that $\cup_{n\in\N} Y_n$ is a conull set. Indeed,
		\begin{equation*}
			\begin{split}
				\mu_\omega(Y_n^c)&=\mu_\omega(\{x \in\partial \beta X; \exists y \in \partial \beta X \text{ such that } d(x,y) < n \Longrightarrow  (x,y)\notin E_n \}) \\
				&= \mu_\omega(\{x \in\partial \beta X; \exists y \in \partial \beta X \text{ such that } d(x,y) < n \Longrightarrow  (x,y)\in \cup_{i\geq n} F_i^c\} )\\
				&\leq \sum_{i\geq n}\mu_\omega(\{x \in\partial \beta X; \exists y \in \partial \beta X \text{ such that } d(x,y) < n \Longrightarrow  (x,y)\in F_i^c\}\\
				&\leq \sum_{i\geq n} \dfrac{1}{2^i} \leq \dfrac{1}{2^{n-1}} \longrightarrow 0.
			\end{split}
		\end{equation*}
		Thus, by definition, ${\partial G(X)}$ is a $\mu_\omega$-hyperfinite equivalence relation over $(\partial \beta X, \mu_\omega)$.\end{proof}

	In a similar way as we did in the Proposition \ref{Elekweakdef} for sequences of graphs, we get another equivalence:

	\begin{cor}\label{Elekcor}
		 The equivalence relation ${\partial G(X)}$ is $\mu_\omega-$hy\-per\-fi\-nite if and only if for all $\varepsilon>0$, there is $K \in \N$ and a subset $Z\subset \partial \beta X$ such that $\mu_\omega(Z)<\varepsilon$ and the connected components of ${\partial G(X)} \upharpoonright Z^c$ have size of at most $K$, where ${\partial G(X)} \upharpoonright Z^c= {\partial G(X)}\cap (Z^c\times Z^c)$ are the edges that are not incident to $Z$. 
	\end{cor}
	\begin{proof}
		
		By the last proposition, given $\varepsilon>0$ and $S=2$, there is a $K\in\N$ and finite subequivalence relation $F\subseteq R_{\partial G(X)}$ such that, $ \mu(\{x \in \partial \beta X ; \diam([x]_F)<K\})> 1- \varepsilon$. Take the set $Z=\partial \beta X \setminus \{x \in \partial\beta X; \diam([x]_F)<K \} $. By construction, $\mu_\omega(Z)<\varepsilon$ and the components of ${\partial G(X)}$ without the edges incidents to $Z$ have size of at most $K$.
		
		On the other hand, we need to construct the chain of finite subequivalence relation of ${\partial G(X)}$.
		Take $\varepsilon=\frac{1}{2^n}$, so there exists $K_n\in \N$ and $Z_n\subseteq\partial\beta X$ such that $\mu_\omega(Z_n)<\frac{1}{2^n}$ and the components of ${\partial G(X)} \upharpoonright Z^c_n$ have size bounded by $K_n$. Define the subrelations
		$$E_0= \Delta_{\partial \beta X}= \{(x,x) \in \partial \beta X\times \partial \beta X\}.$$
		$$E_{n}= {\partial G(X)} \upharpoonright Z^c_1 \cup \dots \cup {\partial G(X)} \upharpoonright Z^c_n = {\partial G(X)}\upharpoonright \cup_{i=1}^n Z_i^c .$$
		
		Clearly, $E_n\subseteq E_{n+1}$. It just rests to prove that the union of $Z_n^c$ is a conull set. Fix $n\in\N$,

		$$ \mu_\omega\bigg( \big(\bigcup\limits_{i=1}^n Z_i^c \big)^c \bigg)=\mu_\omega\bigg( \bigcap^n_{i=1} Z_i \bigg) \leq \mu (Z_n) \leq  \frac{1}{2^n}  \overset{n \to \infty}{\longrightarrow}0.$$

		Thus, $({\partial G(X)},\mu_\omega)$ is $\mu_\omega-$hyperfinite.\end{proof}
	
	As expected, the coarse boundary groupoid takes a hyperfinite sequence of graphs and induces the hyperfiniteness on the equivalence relation and vice versa.

	\begin{thm}\label{hyperhyper}
		Let $\mathcal{X}=\{ X_i\}_{i\in \N}$ be a sequence of finite bounded degree graphs and $X$ the related space of graphs. Then, the sequence $\{ X_i\}_{i\in \N}$ is hyperfinite if and only if for every non-principal ultrafilter $\omega \in \partial\beta \N$, the equivalence relation $( {\partial G(X)}, \hat{\mu}_\omega) $ is $\mu_\omega$-hy\-per\-fi\-ni\-te.
	\end{thm}	
	\begin{proof}
		Given $\varepsilon >0$, by Proposition \ref{Elekweakdef}, there exists $K_{\varepsilon} \in \N$ and 
		$ Z_i \subseteq V(X_i)$ such that
		$ \limsup\limits_{i \rightarrow \infty} \frac{|Z_i|}{|V(X_i)|} \leq \varepsilon $ and all graphs without the incident points of $\{Z_i\}_{i\in \N}$ have components with size smaller than $K_{\varepsilon}$. Fix a non-principal ultrafilter $\omega \in \partial \beta\N$, by Lemma \ref{subsetUltra}, there is a set $Z
		=[Z_i]_\omega \subseteq \partial \beta X$ such that $\mu_\omega(Z) = \lim\limits_{i\to\omega}\mu_i(Z_i)$.  That means, $\mu_\omega(Z)\leq \varepsilon.$
		
		We need to prove that all connected components of $R_{\partial G(X)} \upharpoonright Z^c$ have size smaller than $K_{\varepsilon}$. Suppose that one component, namely $D$, has size bigger than $K_{\varepsilon}$. Take $x \in D\subseteq Z^c$, there is sequence $(x_i)_{i\in \N}\in \{Z_i^c\}_{i\in \N}$ such that $x=\lim\limits_{i\to\omega}x_i$. By Proposition \ref{isometric}, we have an isometry between $\bigg(\prod\limits_{i\to\omega} X_i,x_i\bigg)$ and $\big(\partial G(X)\big)_x $. This implies that the size of the component that $x_i$ belongs to is bigger than $K_{\varepsilon}$, which is a contradiction because $x_i \in X_i\setminus Z_i$. Thus, by Corollary \ref{Elekcor}, $( R_{\partial G(X)}, \hat{\mu}_\omega) $ is $\mu_\omega-$hy\-per\-fi\-ni\-te.

		Conversely, take a non-principal ultrafilter $\omega\in \partial\beta\N$ and suppose that $( R_{\partial G(X)}, \hat{\mu}_\omega) $ is $\mu_\omega-$hy\-per\-fi\-ni\-te. Given $\varepsilon>0$, there is a $K_{\varepsilon}\in\N$ and $Z\subset \partial\beta X$ such that $\mu_\omega(Z)<\varepsilon$ and the components of $E=R_{\partial G(X)}\upharpoonright Z^c$ have size of at most $K_{\varepsilon}$.  We just need to prove that those properties are still holding asymptotically along the non-principal ultrafilter $\omega\in \partial\beta\N$. As $Z\subseteq\partial\beta X$, by Lemma \ref{subsetUltra}, there is a sequence $\{Z_i\}_{i\in \N} \subseteq \{X_i\}_{i\in \N}$, such that $\mu_\omega(Z \triangle [Z_i]_\omega)=0$. Let $\mathcal{X'}$ be the subsequence of graphs obtained from the sequence $\mathcal{X}$ by deleting the edges incident to $\{Z_i\}_{i\in \N}$. By construction, $$\mu_\omega([Z_i]_\omega)=\lim\limits_{i\to\omega} \frac{|Z_i|}{|V(X_i)|} \leq \varepsilon.$$ 
		
		Notice that all components of $\mathcal{X'}$ have size at most $K_{\varepsilon}$. Indeed, for every $x\in X_i$, define $C_x$ as the component that $x\in X_i$ belongs in $\mathcal{X'}$ and $C_i= \{x\in X_i; |C_x| \geq K_{\varepsilon} \}\subseteq \mathcal{X'}$. We claim that $(\mu_i( C_i))_{i\in\N}$ tends to $0$, when $i\to \infty$.
		Suppose the opposite, i.e, there is a $\alpha_0>0$ such that $(\mu_i( C_i))_{i\in\N}$ converges to $\alpha_0$. Again by Lemma \ref{subsetUltra}, there is a $C'=[C_i]_\omega \subseteq \partial\beta X \setminus Z$ with $\mu_\omega(C')=\alpha_0$. Notice that if $x \in C'$ then $|C'_x|\geq K_{\varepsilon_0}$, by the same argument using the isometry. This is a contradiction, since $C'\subseteq \partial\beta X \setminus Z$. 
		
		Thus, by Proposition \ref{Elekweakdef}, the sequence of graphs $\{X_i\}_{i\in \N}$ is hyperfinite along a non-principal ultrafilter $\omega \in \partial \beta \N$. Since $( R_{\partial G(X)}, \hat{\mu}_\omega) $ is $\mu_\omega-$hy\-per\-fi\-ni\-te, for every non-principal ultrafilter $\omega \in \partial\beta \N$, by Lemma \ref{hyperultra}, we have that the sequence of graphs $\{X_i\}_{i\in \N}$ is hyperfinite.\end{proof}

	\subsubsection{Other characterizations of hyperfiniteness}

	The Connes--Feldman--Weiss Theorem is a known equivalence between $\mu-$hyperfinteness and measurable amenability of a countable Borel equivalence relation on a standard probability space. Another famous characterization of hy\-per\-fi\-ni\-te\-ness is due to Kai\-ma\-novich, that calculates the isoperimetric constant of the sets. In this section, we define the isoperimetric constant in the set of equivalence relations and present a modified version of both theorems. We said modified because we will restrict to the coarse boundary groupoid where the measure is on a non-standard probability space $(\partial\beta X, \mu_\omega)$ but the result still holds. The proof of the classical Theorems can be found in \cite{marks2017short}, where Marks uses the Connes-Feldman-Weiss Theorem to prove Kaimanovich's Theorem.

	By the theory developed in \cite{carderi2018non}, it is easy to see that the standard Reiter's condition of measurable amenability (as we defined in Definition \ref{AmeGroupoid}) is applicable to our coarse boundary groupoid $(\partial G(X),\hat{\mu}_\omega)$ even over the non-standard space $(\partial \beta X, \mu_\omega)$, where $X$ is the space of finite bounded degree graphs. As $\partial G(X)$ and $R_{\partial G(X)}$ coincide, we interchange the notation for convenience. 
	
	\begin{thm}[Connes--Feldman--Weiss theorem for the coarse boundary groupoid] \label{CFW}
		Let $R_{\partial G(X)}$ be the countable equivalence relation on $(\partial\beta X,\mu_\omega)$ for some non-principal ultrafilter $\omega \in \partial\beta \N$. Then, the equivalence relation $R_{\partial G(X)}$ on $(\partial \beta X, \mu_\omega)$ is $\mu_\omega$-hyperfinite if and only if the coarse boundary groupoid $({\partial G(X)},\hat{\mu}_\omega)$ is measurably amenable.
	\end{thm}
	\begin{proof}
		Let $R_{\partial G(X)}$ be $\mu_\omega-$hyperfinite, then there exists a nested sequence $E_0\subseteq E_1\subseteq \dots $ of finite subequivalence relations such that $R_{\partial G(X)}=\cup_{n\in\N} E_n$, up to null sets. For every $n\in \N$, define $\phi_n: \partial \beta X \to \ell^1 (\partial G (X), \lambda^x )$ by $$ 
		x \longrightarrow \phi_n^x(y)= \bigg\{\begin{array}{cc}
			0, & \text{if } y \notin [x]_{E_n} \\
			\frac{1}{|[x]_{E_n}|}, & \text{if } y \in [x]_{E_n}.
		\end{array}\bigg. $$	
		It is easy to see that $\phi_n$ is normalized and the sequence $ |\phi_n^x(a)-\phi_n^y(a)|$ tends to zero, when $n$ goes to infinity, since $E_n$ is growing. Thus, the equivalence relation is measurably amenable.

		Conversely, following the ideas of \cite{carderi2018non}, one can see that the coarse boundary groupoid $({\partial G(X)},\hat{\mu}_\omega)$ is a p.m.p. groupoid over the probability space $(\partial \beta X,\mu_\omega)$, that is, a groupoid equiped with a $\sigma-$algebra $B$ and a measure $\hat{\mu}_\omega$ such that the fiber spaces $(\partial G(X), B,\hat{\mu}_\omega,s)$ and $(\partial G(X), B,\hat{\mu}_\omega,r)$ are countable with measurable inverse map and measurable multiplication map. In particular, it is realizable because it comes from a sequence of finite graphs with bounded degree. By Theorem 3.28 of \cite{carderi2018non}, the coarse boundary groupoid admits a p.m.p factor $(E,\mu')$ on a standard Borel probability space that is weakly-equivalent to $(\partial G(X),\hat{\mu}_\omega)$.
		
		If the coarse boundary groupoid $({\partial G(X)},\hat{\mu}_\omega)$ is measurably amenable, we have that the factor map is a homomorphism that maps $\hat{\mu}_\omega$ to $\mu'$ and commutes with source, range, product and inversion map, then the factor $(E,\mu')$ is amenable. Applying the classical Connes-Feldman-Weiss Theorem, $(E,\mu')$ is $\mu'$-hyperfinite. By the weak-equivalence with the factor, we have that $(R_{\partial G(X)},\hat{\mu}_\omega)$ is $\mu_\omega$-hyperfinite.\end{proof}

	\subsubsection{Property A on average}

	Inspired by the definition of measurable amenability on the coarse groupoid, we define the related coarse property along the sequence of finite graphs. We called it \emph{property A on average}. It turns out that this property is weaker than Property A for sequence of graphs, but it seems to be the sufficient condition to imply the measurable amenabilty of the coarse boundary groupoid.  In the end, this property is also equivalent to property almost-A and hyperfiniteness. 
	
	\begin{defi} Let $\mathcal{X}=\{X_i\}_{i \in \N}$ be a sequence of finite bounded degree graphs where $X$ the related space of graphs. A sequence of graphs $\{X_i\}_{i\in\N}$ has \emph{property A on average}, if for every $\varepsilon>0$, there exists $S>0$ and a map $\xi:X \to \ell^1(X)$ such that:
		\begin{enumerate}[$(1)$]
			\item $||\xi_x ||_1=1$, for all $x\in X$;
			\item For all $x\in X$, we have $\supp(\xi_x) \subseteq B_S(x)$, i.e, the support of $\xi_x$ is contained in an ball of radius $S$ around $x$ ;
			\item and  $$\limsup_{i \rightarrow \infty} \dfrac{1}{|X_i|} \sum_{x \in X_i} \sum_{ \substack{y \in X_i \\ (x,y)\in E(X_i); }} ||\xi_x - \xi_y ||_1 < \varepsilon. $$
		\end{enumerate} 
	\end{defi}
	
	Notice that the only difference between property A (Definition \ref{propA}) and property A on average is the respective third condition. It is clear that property A implies property A on average. Similar to the hyperfiniteness, we have an equivalent definition along a non-principal ultrafilter $\omega \in \partial \beta \N$
	
	\begin{defi}\label{onavg}
		Let $\mathcal{X}=\{X_i\}_{i \in \N}$ be a sequence of finite bounded degree graphs, $X$ the related space of graphs and a non-principal ultrafilter $\omega \in\partial\beta \N$. The sequence $\mathcal{X}$ of graphs has \textit{property A on average along the non-principal ultrafilter $\omega \in \partial \beta \N$}, if for every $\varepsilon>0$, exists $S>0$ and a map $\xi:X \to \ell^1(X)$ such that:
		\begin{enumerate}[$(1)$]
			\item $||\xi_x ||_1=1$, for all $x\in X$;
			\item For all $x\in X$, we have $\supp(\xi_x) \subseteq B_S(x)$, i.e, the support of $\xi_x$ is contained in an ball of radius $S$ around $x$ in the space of graphs $X$;
			\item and  $$ \lim_{i\to\omega} \dfrac{1}{|X_i|} \sum_{x \in X_i} \sum_{ \substack{y \in X_i \\ (x,y)\in E(X_i) }} ||\xi_x - \xi_y ||_1 < \varepsilon. $$
		\end{enumerate} 
	\end{defi}

	The proof of the equivalence between $\mathcal{X}=\{X_i\}_{i\in\N}$ has property A on average and property A on average along every non-principal ultrafilter $\omega \in \partial \beta \N$ is similar to Lemma \ref{hyperultra}. The next proposition justifies the motivation behind property A on average, since the amenability on the coarse boundary groupoid implies this coarse property on the space of graphs.

	\begin{lemma}\label{AmenaONavg}
		Let $\mathcal{X}=\{X_i\}_{i\in\N}$ be a sequence of finite bounded degree graphs and $\partial G(X)$ the respective boundary groupoid.
		If the coarse boundary groupoid $({\partial G (X)},\hat{\mu}_\omega)$ is measurable amenable for a non-principal ultrafilter $\omega \in \partial \beta \N$, then $\mathcal{X}$ has property A on average along $\omega\in\partial\beta\N$. 
	\end{lemma}
	\begin{proof}
		given a non-principal ultrafilter $\omega \in\partial \beta \N$ ans $\varepsilon>0$, for some $n\in\N$, exists a map $\eta\coloneqq \phi_n\in L^\infty(\partial\beta X,\ell^1(\partial G(X),\lambda^x))$ from the amenability of $\partial G(X)$ such that for every $x\in\partial G (X)$
		
		$$ \sum_{y\in G^{x}} \sum_{z\in G^{s(y)}} f(x) |\eta(yz)-\eta(z)| \leq \varepsilon.$$

		By Lemma \ref{appfunc},  we can approximate the map $\eta$ by a sequence of maps $\eta_i \in L^\infty(X_i, \ell^1(X_i\times X_i,\lambda^{x_i}))$ such that $\lim\limits_{i\to\omega}\eta_i=\eta$. Even more, if we write $y= [(x_i,a_i)]_\omega$, $z=[(a_i,b_i)]_\omega$, we have	
		$$\varepsilon \geq  \lim_{i\to\omega} \dfrac{1}{|X_i|} \sum_{x_i\in X_i}\sum_{a_i\in X_i} \sum_{b_i\in X_i} f_i(x_i) |\eta_i(x_i,b_i)-\eta_i(a_i,b_i)|$$ 
		
		Define $\eta_i (x_i,b_i)= \colon \xi_{x_i} (b_i)$ and take $f_i$ to be an approximation of the constant function equals to one.
		
		\begin{equation*}
			\begin{split}
				\varepsilon& \geq \lim_{i\to\omega} \dfrac{1}{|X_i|} \sum_{x_i \in X_i}  \sum_{a_i\in X_i} \sum_{b_i\in X_i} | \xi_{x_i} (b_i)- \xi_{a_i} (b_i)|\\	
				& = \lim_{i\to\omega} \dfrac{1}{|X_i|}  \sum_{x_i \in X_i}  \sum_{a_i \in X_i}   ||\xi_{a_i} - \xi_{x_i} ||_1 \\
				&  = \lim_{i\to\omega} \dfrac{1}{|X_i|}  \sum_{x_i \in X_i}  \sum_{ \substack{a_i \in X_i \\ (x_i,a_i)\in E(X_i); }} ||\xi_{a_i} - \xi_{x_i} ||_1.
			\end{split}
		\end{equation*}
		
		This completes the proof, since ${\xi_x}_i $ is normalized by the normalization of $\eta$, clearly, the sequence of maps ${\xi_x}_i \in \ell^1(X_i)$ can be taken, without lost of generality, such that they are uniformly supported in balls of radius $S$.
		Thus, the map $\xi$ satisfies the properties of Definition \ref{onavg}. Hence, $\mathcal{X}$ has Property A on average along the non-principal ultrafilter $\omega\in\partial\beta \N$.\end{proof}


	This Lemma show us that the measurable amenability of the coarse boundary groupoid is a sufficient condition for the space of graphs have property A on average. In \cite{kaiser2019combinatorial}, Kaiser introduced another coarse property along of the sequence of graphs that easily implies that the coarse boundary groupoid is measurable amenable. In his paper, he introduces the following property:

	\begin{defi}\label{propertyalmostAdef}
		A sequence $\mathcal{X}=\{X_i\}_{i\in\N}$ of finite bounded degree graphs has \textit{property almost-A} if there exist subsets $V_i \subset V(X_i)$ such that $\lim\limits_{i \rightarrow \infty }\frac{|V_i|}{|V(X_i)|}=0$ and $\{X_i\backslash V_i\}_{i\in\N}$ has property A.
	\end{defi}
	
	Take $X$ to be the space of graphs of the sequence of graphs $\{X_i\}_{i\in\N}$ with property almost-A and $X'$ the space of graphs related to the sequence the subgraphs $\{X_i\backslash V_i\}_{i\in\N}$ that have property A.  Using our notation of the coarse boundary groupoid, it is easy to see that $\partial G(X)$ and $\partial G(X')$ are measurable isomorphic, since they come from the same graphs except for this sequence of subgraphs with measure zero along the limit. Thus, both groupoids are measurable amenable, since $X'$ have property A. In this short remark we didn't mention the dependency on the ultrafilter, but it follows easily if we consider the property A along every non-principal ultrafilter.
	
	\begin{lemma}\label{almostaultra}
		A sequence $\mathcal{X}=\{X_i\}_{i\in\N}$ of finite bounded degree graphs has property almost-A if and only if  $\mathcal{X}$ has property almost-A along every non-principal ultrafilter $\omega\in\partial\beta\N$, i.e, for every non-principal ultrafilter $\omega \in\partial \beta \N$ and for every $i\in \N$, there exist subsets $V_i \subset V(X_i)$ such that $\lim\limits_{i \rightarrow \omega }\frac{|V_i|}{|V(X_i)|}=0$ and the sequence $\{X_i\backslash V_i\}_{i\in \N}$ has property A.
	\end{lemma}
	\begin{proof}
		Obviously, if $\mathcal{X}$ has property almost-A then has property almost-A along every non-principal ultrafilter $\omega \in\partial \beta \N$. Suppose the opposite for the converse, i.e, there is an $\varepsilon'>0$ where all subsets $V_i\subset V(X_i)$ such that $\frac{|V_i|}{|V(X_i)|}\leq \varepsilon'$, the sequence $ \mathcal{X'}=\{X_i\setminus V_i\}_{i\in\N}$ don't have property A. That means, there is an $\varepsilon_0>0$ that for all $S\in\N$ and all normalized function $\eta: X'\to \ell^1(X')$ with $\supp(\eta_x)\subseteq B_S(x)$ satisfies $||\eta^i_x-\eta^i_y ||_1\geq \varepsilon_0$ for all $(x,y)\in E(X'_i)$, where $X'$ the space of graphs of the sequence $\{X_i\backslash V_i\}_{i\in\N}$.
		
		For $\varepsilon_0$, take a natural sequence of $S $ going to infinity, then, there is a sequence of index $i_S$ going to infinity such that all normalized functions $\eta^S: X'_{i_S}\to \ell^1(X'_{i_S})$ with $\supp(\eta^S_x)\subseteq B_{S}(x)$ such that $||\eta^S_x-\eta^S_y ||_1\geq \varepsilon_0$, for all $(x,y)\in E(X'_{i_S})$.
		Clearly, we have an infinite set $I_S=\{i_S\in\N\}$ satisfying above conditions, thus, there is a non-principal ultrafilter $\omega \in \partial \beta \N$ that contains $I_S$. We have a contradiction with the fact that $\{X_i\}_{i\in \N}$ has property almost-A along every non-principal ultrafilter $\omega\in\partial \beta \N$.\end{proof}
	
	To prove that property A on average and property almost-A are equivalent, we use a method introduced in \cite[Theorem 4.4]{kaiser2019combinatorial}.

	\begin{lemma}[Kaiser's removing points method]\label{Removing}
		Let $X$ be finite bounded degree graph with pro\-per\-ty A and $\xi: X \to \ell^1(X)$ is the map satisfying the conditions of property A for some fixed $\varepsilon$ and $S$. Given a proper subgraph $W \subset X$, there exists $\xi': X\setminus W \to \ell^1(X\setminus W)$ satisfying the conditions of property A for $\varepsilon$ and $M = \max\limits_{y \in X} |B_S(y)| $, that means, $X\setminus W$ still has property A.
	\end{lemma}

	\begin{rem}
	The above lemma appeared as part of \cite[Theorem 4.4]{kaiser2019combinatorial}, where Kaiser proves that a sequence of finite bounded degree graphs has property A then it is also hyperfinite.We remark here that as a byproduct of our results one obtains an easier way to prove it. Indeed, by {\cite[Theorem 5.3]{skandalis2002coarse}}, property A for $X$ implies that $\partial G(X)$ is topologically amenable. Thus, $(\partial G(X),\hat{\mu}_\omega)$ is measurably amenable, for every non-principal ultrafilter $\omega \in \partial \beta \N$. Now, by Theorem \ref{CFW}, the induced equivalence relation $({\partial G(X)},\hat{\mu}_\omega)$ is hyperfinite and thus the sequence of graphs $\{X_i\}_{i\in\N}$ is hyperfinite. 
	\end{rem}
	Most of the results from \cite{kaiser2019combinatorial} can be described in terms of the coarse boundary groupoid. Moreover, with the above method we can prove that the property almost-A matchs our definition of property A on average.

	\begin{prop}\label{OnAvgAlmostA}
		A sequence of finite bounded degree graphs $\mathcal{X}=\{X_i\}_{i\in\N}$ has property A on average along a non-principal ultrafilter $\omega\in\partial\beta\N$ if and only if $\mathcal{X}$ has property almost-A along the non-principal ultrafilter $\omega \in \partial \beta \N$.
	\end{prop}
	\begin{proof}
		Given a $k\in \N$. Suppose that the sequence of finite graphs $\mathcal{X}=\{X_i\}_{i\in\N}$ has property A on average and take $\varepsilon = \frac{1}{k^2}$, such that there are $S(k)$ and a normalized map $\xi^k:X \to \ell^1(X)$ with $\xi^k_x$ supported in a ball of radius $S(k)$ around $x\in X$ that satisfies  
		
		$$ \lim\limits_{i \to \omega} \dfrac{1}{|X_i|} \sum_{x \in X_i} \sum\limits_{ \substack{y \in X_i \\ (x,y)\in E(X_i); }} ||\xi^k_x - \xi^k_y ||_1 <\frac{1}{k^2}. $$
		Take $C_k=\frac{1}{k}$ and define $Z_{i}^k= \bigg\{ x \in X_i ; \sum\limits_{ \substack{y \in X_i \\ (x,y)\in E(X_i) }} ||\xi^k_x - \xi^k_y ||_1> C_k=\frac{1}{k}\bigg\} \subseteq X_i$.
		
		Notice that, for some $i\in \N$,
		\begin{align*}
			\dfrac{C_k |Z_{i}^k|}{|X_i|} &\leq \dfrac{1}{|X_i|} \sum_{x \in Z_{i}^k} \sum\limits_{ \substack{y \in X_i \\ (x,y)\in E(X_i) }} ||\xi^k_x - \xi^k_y ||_1 \\
			&\leq \dfrac{1}{|X_i|} \sum_{x \in X_i} \sum\limits_{ \substack{y \in X_i \\ (x,y)\in E(X_i) }} ||\xi^k_x - \xi^k_y ||_1\\
			& < \frac{1}{k^2}.
		\end{align*}
		That means, there is an $I_k\in \N$ such that for all $i\geq I_k$, we have $\dfrac{|Z_{i}^k|}{|X_i|}  < \dfrac{\frac{1}{k^2} }{\frac{1}{k}} = \frac{1}{k}$.

		Now, for any $\varepsilon>0$, we consider $N= \big\{k \in \N ; \varepsilon\geq \frac{1}{k^2} \big\} \subseteq \N$ and apply the previous construction for all $k \in N$. For every $i\in \N$, define $k_i=\max\{k \in N; I_k < i \}$. We consider $Z_i \colon = Z_i^{k_i}$. Notice that $\frac{|Z_{i}|}{|X_i|}$ is bounded by $\frac{1}{{k_i}}$ and clearly tends to $0$, when $i$ goes along $\omega \in \partial \beta \N$. Denote the restriction of $\xi^{k_i}$ to the complement of $Z_i$ by $\eta^{k_i}: X_i\setminus Z_i\to \ell^1(X_i)$. Let's prove that $\eta^{k_i}$ satisfies the conditions of property A for all $i\in \N$. It is easy to see that the conditions $(i)$ and $(ii)$ of the definition of property A on average along $\omega \in \partial \beta \N$ implies conditions $(i)$ and $(ii)$ of property A, with $S=\max S(k_i)$. We only need to verify the third condition.
		
		For $(x,y)\in E(X_i\setminus Z_i)$ and all $i\in \N$, we have
		$$||\eta^{k_i}_x - \eta^{k_i}_y ||_1 \leq  \sum\limits_{ \substack{y \in X_i\setminus Z_i \\ (x,y)\in E(X_i) }} ||\xi^{k_i}_x - \xi^{k_i}_y ||_1 \leq C_{k_i} \overset{i \to \infty}{\longrightarrow} 0$$

		Now we can apply Lemma \ref{Removing}, then $\eta^{k_i}: X_i\setminus Z_i\to \ell^1(X_i\setminus Z_i)$ satisfies the conditions of property A for all $i\in \N$. Hence, $\{X_i\}_{i\in\N} $ has property almost-A along $\omega\in\partial\beta\N$.

		The converse follows since the property almost-A of the sequence of graphs $ \{X_i\}_{i\in\N}$ implies that the coarse boundary groupoid $(\partial G(X),\hat{\mu}_\omega)$ is measurably amenable for every non-principal ultrafilter $\omega \in \partial \beta \N$. Thus, by Lemma \ref{AmenaONavg}, the sequence of graphs $\{X_i\}_{i\in\N}$ has property A on average. 	\end{proof}
	
	We are now able to deduce all implications in the following theorem.	
	
	\begin{thm}\label{mainthm}
		Let $\mathcal{X}=\{X_i\}_{i\in\N}$ be a sequence of finite graphs with bounded degree such that the cardinality of $X_i$ goes to infinity when $i$ goes to infinity. Take $X$ to be the space of graphs of this sequence and ${\partial G(X)}$ the related coarse boundary groupoid. The following statements are equivalent:
		\begin{enumerate}[$(1)$]
			\item The measurable coarse boundary groupoid $({\partial G(X)},\hat{\mu}_\omega)$ is measurably amenable, for every non-principal ultrafilter $\omega\in\partial\beta\N$;
			\item  The sequence of graphs $\mathcal{X}$ has property A on average along every non-principal ultrafilter $\omega\in\partial\beta\N$;
			\item The sequence of graphs $\mathcal{X}$ has property almost-A along every non-principal ultrafilter $\omega \in \partial \beta \N$;
			\item The sequence of graphs $\mathcal{X}$ is hyperfinite;
			\item The measurable equivalence relation $(R_{\partial G(X)},\hat{\mu}_\omega)$ induced by the coarse boundary groupoid is $\mu_\omega$-hyperfinite, for every non-principal ultrafilter $\omega\in\partial\beta\N$.
		\end{enumerate}
	\end{thm} 
	
	\begin{proof} 
		The implication $(1)$ to $(2)$ follows from Lemma \ref{AmenaONavg}. The conditions $(2)$ and $(3)$ are equivalent by the Proposition \ref{OnAvgAlmostA}. As we commented before the last Proposition, $(3)$ implies $(4)$.  Finally, Theorem \ref{hyperhyper} implies that $(4)$ and $(5)$ are equivalent and the modified version of Connes-Feldman-Weiss's Theorem (Theorem \ref{CFW}) covers the equivalence between $(5)$ and $(1)$.	\end{proof}

	\subsection{A-T-menability}
	
	We know that a space of bounded degree graphs $X$ is coarsely embeddable into a Hilbert space if and only if the coarse groupoid $G(X)$ is a-T-menable. In this case the reduction $\partial G(X)$ is also a-T-menable. We saw in Proposition \ref{Rufus}, that the space of graphs $X$ be asymptotically coarsely embeddable into a Hilbert space is a sufficient condition for the coarse boundary groupoid $\partial G(X)$ be a-T-menable. In this section we prove the converse of this result.
	
	This section consists of two parts. In the first one assume that the coarse boundary groupoid $\partial G(X)$ is topologically a-T-menable; using the Tietze's Theorem, we lift the conditionally negative definite function $\psi$ to the sequence of graphs and by some spectral analysis we deduce that after a small perturbation, this lift can be forced to satisfy the conditions of asymptotic coarse embeddability into a Hilbert space. In the second part, we suppose that the coarse boundary groupoid $\partial G(X)$ is measurably a-T-menable; using the same ideas of the topological case, we prove that the sequence of graphs then satisfies a weak form of an asymptotic coarse embeddability into a Hilbert space: that is, after discarding spaces of small measure from the sequence of graphs, the remaining space of graphs equipped with the old metric becomes asymptotically coarsely embeddable into a Hilbert space. 

	\subsubsection{Asymptotic coarse embeddability}

	In order to make the ideas more comprehensible, we prove some preparatory lemmas beforehand. For the first lemma, using the properties of the coarse boundary groupoid $\partial G(X)$, we investigate the properties of an approximation of the conditionally negative definite function $\psi$ defined on the coarse boundary groupoid.

	\begin{lemma} \label{Lemma1} Let $X$ be the space of graphs of a sequence of finite bounded degree graphs $\{ X_i\}_{i\in\N}$. If $\psi: \partial G(X) \to \R$ is a continuous function, then there exist a sequence of symmetric limit-normalized kernels $(K_i: X_i \times X_i \to \R )_{i\in \N}$ that represent  $\psi$ on the boundary and non-decreasing control functions $\rho_1,\rho_2:\R_+ \to \R$ for the kernels $K_i$, i.e, for all $x_i,y_i \in X_i$, $$\rho_1(d(x_i,y_i))\leq K_i(x_i,y_i)\leq \rho_2(d(x_i,y_i)).$$ Moreover, if $\psi$ is a proper function then the control functions go to infinity at infinity, i.e, $\lim\limits_{r\to \infty} \rho_i(r)=\infty$, for $i=1,2$.
		
	\end{lemma} 
	
	\begin{proof}
		Take  a continuous function $\psi\in C( \partial G(X))$. Since $\partial G (X)$ is a closed subset of $G (X)$, by Tietze's Theorem, we can extend $ \psi$ continuously to $\psi': G(X) \to \R $ such that $\psi'|_{\partial\beta X}= \psi$. As $G(X) \subseteq X\times X \cup\partial\beta X \times \partial \beta X$, we can see $\psi': G(X) \to \R $ as a sequence of functions $(K_i': X_i \times X_i \to \R)_{i\in\N}$  that approximates $\psi'$. In particular, $(K_i')_{i\in\N}$ is an approximation of $\psi$ on the boundary. Even more, by continuity and limit properties, the sequence $(K_i')_{i\in\N}$ can be taken as symmetric limit-normalized functions.

		Our goal is to construct the  control function $\rho_1$. It will be necessary to do some rearrangements on the sequence of kernels to make it clear that they are controllable by the functions $\rho_1$ and $\rho_2$ along the limit. 
		
		For some $m\in \N$, we define
		$$\rho_1^{K',m}(\ell)\coloneqq \inf\limits_{\substack{d(x_i,y_i)=\ell\\ (x_i,y_i)\in X_i\\i\geq m}} K_i'(x_i,y_i).$$ When $m$ tends to infinity the sequence $\rho_1^{K',m}(\ell)$ converges to $ \rho_1^\psi(\ell) \coloneqq \inf\limits_{d(x,y)=\ell} \psi(x,y)$. By this fact, for every $\ell$, there is a $n \in \N$ such that $\rho_1^{K',m}(\ell) \geq \frac{\rho_1^\psi (\ell)}{2}$, for all $m\geq n$. For the sequence of kernels $K'=(K_i')_{i\in\N}$, we take $$K_i''(x_i,y_i)=\bigg\{\begin{array}{cc} 
			\frac{\rho_1^\psi (d(x_i,y_i))}{2} , & \text{if } K_i'(x_i,y_i)\leq  \frac{\rho_1^\psi (d(x_i,y_i))}{2} \\
			K_i'(x_i,y_i), & \text{otherwise.}
		\end{array}$$

		Notice that the limit of $(K_i'')_{i\in \N}$ also represents $\psi$ on the boundary, because for every $R>0$, there is a $m$ sufficiently large such that $K_i''\big|_{E_R^i} =K_i'\big|_{E_R^i}$ for all $i\geq m$. In particular, $(K_i'')_{i\in \N}$ represent $\psi$ in $\overline{E_R}$ and thus in $\partial G(X)= \bigcup\overline{E_R} \cap \partial\beta X \times \partial\beta X $. Define $$\rho_1^{K''}(\ell) \coloneqq \inf\limits_{\substack{ d(x_i,y_i)=\ell\\ (x_i,y_i)\in X_i }} K_i''(x_i,y_i).$$ Clearly,   $\rho_1^{K''}(d(x_i,y_i)) \leq K_i''(x_i,y_i)$. Moreover, if $\psi$ is proper, then $\rho_1^{K''}(\ell)$ goes to infinity when $\ell$ increases since for an infinite sequence $(\gamma_i)_{i\in\N}$ in $\partial G(X)$ that escapes to infinity, the sequence $(\psi(\gamma_i))_{i\in\N}$ also escapes to infinity.

		To express $\rho_2$ we apply an analogous argument taking the supremum of the kernels.  It is important to highlight that the infimum and supremum values of $\psi$ are always attained in the compact set $\overline{E_\ell}$. For $\rho_2^\psi(\ell) \coloneqq \sup\limits_{d(x,y)=\ell} \psi(x,y)$, define
		
		$$K_i(x_i,y_i)=\bigg\{\begin{array}{cc} 
			\frac{\rho_2^\psi (d(x_i,y_i))}{2} , & \text{if } K_i''(x_i,y_i)\leq  \frac{\rho_2^\psi (d(x_i,y_i))}{2} \\
			K_i''(x_i,y_i), & \text{otherwise.}
		\end{array}$$

		By the same argument, we see that $(K_i)_{i\in\N}$ approximates $\psi$. Moreover, it is also limit-normalized and symmetric. Thus, for $$\rho_2^{K}(\ell) \coloneqq \sup\limits_{\substack{ d(x_i,y_i)=\ell\\ (x_i,y_i)\in X_i }} K_i(x_i,y_i) \hspace{0.4cm} \text{   and   } \hspace{0.4cm}  \rho_1^{K}(\ell) \coloneqq \inf\limits_{\substack{ d(x_i,y_i)=\ell\\ (x_i,y_i)\in X_i }} K_i(x_i,y_i)$$ we have $$\rho_1^K(d(x_i,y_i)) \leq K_i(x_i,y_i)\leq \rho_2^K(d(x_i,y_i)).$$ 
		
		Notice that $\rho_2$ goes to infinity at infinity, by the fact that $\rho_1$ goes to infinity. Not necessarily $\rho_1^K$ and $\rho_2^K$ are non-decreasing. It depends on $\psi$, but since both functions are growing we can take $\rho_1(r)=\inf\limits_{l\geq r} \rho_1^K(l)$ and analogously for $\rho_2(r)$.\end{proof}
		
	The problem that appears now is the fact that the lifted sequence might stop to be conditionally negative definite; however, as it approximates a conditionally negative definite kernel, it is possible to control this violation precisely, and for that we will estimate the spectrum of the lifted kernels.
	
	For a fixed $i\in \N$ and $R\in \R_+$, we consider for every $x \in  X_i$ the orthogonal projection $$P_x^{R}:\ell^2(X_i)\to \ell^2(B_{R}(x))_0\coloneqq \bigg\{f :B_{R}(x)\to \R; \sum\limits_{y\in B_{R}(x)} f(y)=0\bigg\}.$$ We denote the operator induced by $K_i$ as $T_i: \ell^2(X_i)\to \ell^2(X_i)$ and $\{\delta_{x_j}\}_{j=1}^{|X_i|}$ the basis of $\ell^2(X_i)$. Notice that for all $x\in X_i$ and $R\leq R_i$ we have
	\begin{equation*}
		\begin{split}
			\sigma(P_x^{R} T_i P_x^{R}) \subseteq (-\infty, 0]
			& \iff   \langle P_x^{R} T_i P_x^{R}f,f \rangle \leq 0, \hspace{0.3cm}\forall f \in \ell^2(X_i)\\ 
			&\iff \langle T_i f, f \rangle \leq 0, \hspace{1.4cm}\forall f \in \ell^2(B_{R}(x))_0
		\end{split}
	\end{equation*}
	
	That is, for all subsets $\{x_1, \dots, x_n\}\subseteq B_R(x)$, setting $\lambda_j=f(x_j)$ we have $$\sum_{j=1}^n \lambda_j = 0 \implies \sum_{j,k=1}^n\lambda_j\lambda_k K_i(x_j,x_k)\leq 0.$$

	Therefore, $(K_i)$ is an $(R_i)$-locally conditionally negative definite kernels (see Definition \ref{asympEmb}) if and only if for all $x\in X_i$ and $R\leq R_i$, the spectrum $\sigma(P_x^{R} T_i P_x^{R})$ is non-positive. The next lemma then shows that all approximations $(K_i)_{i\in\N}$ of a conditionally negative definite function $\psi$ are \emph{$b_i$-almost $R_i$-locally conditionally negative definite}, where $b_i$ is a real positive sequence going to $0$: that is, for a fixed $R$, the spectrum $\sigma(P_x^{R} T_i P_x^{R})$ is inside the interval $(-\infty,b_i]$, for all $x\in X_i$. This is expected since we are approximating a conditionally negative definite function.

	\begin{lemma} \label{Lemma2}
		Let $\psi : \partial G(X)\to \R$ be a continuous conditionally negative definite function with the respective sequence of kernels $(K_i)_{i \in \N}$ as in Lemma \ref{Lemma1}. Then for all $R \in \R_+$, there is a sequence of non-negative real numbers $(b_i)_{i\in\N}$ going to zero when $i$ goes to infinity and for every $i\in \N$,
		$$ \bigcup\limits_{x \in X_i} \sigma \bigg(P_x^R T_iP_x^R\bigg) \subseteq (-\infty, b_i].$$

	\end{lemma}
	
	\begin{proof}
		
		Fix some $R\in \R_+$, define $b_i({x_j})= \sigma (P_{x_j}^R T_iP_{x_j}^R)$. For each $i\in \N$, take the maximum values of zero or $b_i({x_j})$ over all ${x_j} \in X_i$ for $1\leq j\leq |X_i|$, i.e, $$b_i\coloneqq \max \{0, \max\limits_{{x_j} \in X_i} b_i({x_j}) \} .$$
		
		Notice that the sequence $(b_i)_{i\in\N}$ tends to zero when $i$ goes to infinty, since the sequence of kernels $ (K_i)_{i\in\N}$ approximates a conditionally negative definite function $\psi$. Indeed, suppose that the sequence $(b_i)_{i\in\N}$ does not converge to zero. Then, there exist $\varepsilon_0>0$ such that $b_i\geq \varepsilon_0$ for all $i$ bigger than some $N$. Notice that, for all ${x_j}\in X_i$, by the definition of $b_i({x_j})$, given any $x_1^j,\cdots, x_n^j \in B_R(x_j)$ and $\lambda_1^j,\cdots, \lambda_n^j\in\R$ such that $\sum\limits_{k=1}^n \lambda_k^j=0$ we have that $\sum\limits_{k,l=1}^n\lambda_k^j\lambda_l^j K_i(x_k^j,x_l^j)\geq \varepsilon_0$. Take some non-principal ultrafilter $\omega\in\partial \beta \N$ and some point $x\in\partial \beta X$ such that $x=\lim\limits_{j \rightarrow \omega}x_j$. We can do the above construction for all graphs, where we can choose sequences of points $x_1^j,\cdots, x_n^j \in B_R(x_j)$, for every $j\in\N$. Notice that the limit of the points $x_1^j,\cdots, x_n^j$ along the ultrafilter lays in the ball of radius $R$ centred in $x$. More over, $[x_1^j]_\omega,\cdots, [x_n^j]_\omega$ is in the fiber $\partial G(X)^x$.  For the sequence of real numbers  $\lambda_1^j,\cdots, \lambda_n^j\in\R$, we know that $\sum\limits_{k=1}^n \lambda_k^j=0$, thus, every sequence $\{\lambda^j_k\}_{j\in\N}$ is bounded, for every $1\leq k\leq n$ and therefore, for some subsequence, it converges to a real number $\lambda_k=\lim\limits_{j\to\omega'}\lambda^j_k$ such that $\lim\limits_{j \to \omega'}\sum\limits_{k=1}^n \lambda_k^j =\sum\limits_{k=1}^n \lambda_k=0$. Then, 
		$$\lim\limits_{j\to \omega'}\sum\limits_{k,l=1}^n\lambda_k^j\lambda_l^j K_i(x_k^j,x_l^j)=\sum\limits_{k,l=1}^n\lambda_k\lambda_l \psi([x_k]_{\omega'},[x_l]_{\omega'})\geq \varepsilon_0$$
		
		This contradict the fact that the funtion $\psi$ is conditionally negative definite. Thus, we can say that the sequence $(b_i)_{i\in\N}$ tends to zero.	
		
		We can conclude by construction that $\sigma (P_x^R T_iP_x^R) \subseteq (-\infty, b_i]$, for all $x \in X_i$. Thus, the union over all $x\in X_i$ of those spectrums lay in the interval $ (-\infty, b_i]$.\end{proof}

	If the approximation $(K_i)_{i\in\N}$ satisfies the lemma above for a sequence $b_i = 0$, then the sequence of kernels $(K_i)_{i\in\N}$ is indeed $R$-locally conditionally negative definite. The next lemma shows us that given any approximation of $\psi$, we can ``shift the approximation by $b_i$'', still remaining close to the previous one and ensuring that the new approximation is $R$-locally conditionally negative definite. 
	

	\begin{lemma} \label{Lemma3}
		
		Keep the assumptions of Lemma \ref{Lemma2}. Given $R\in \R_+$, there is a sequence of kernels $(K_i')_{i\in \N}$ that is $R$-locally conditionally negative definite and  exist a real sequence $(C_{i,R})_{i \in \N}$ converging to zero such that $K_i$ is $C_{i,R}$ close to $K_i'$, i.e, $\max\limits_{x,y\in X_i} |K_i'(x,y) -K_i(x,y) | < C_{i,R}$. In particular, the sequence of kernels $(K_i')_{i \in \N}$ approximates the map $ \psi$.
	\end{lemma}
	
	\begin{proof}	
		Fix $X_i$ and $R>0$, we can exhibit a cover of each $X_i$ with a finite number of balls $\{B_{R}(x_j)\}_{j=1}^{|X_i|}$, i.e, it is the cover by balls of radius $R$, centred in every point of $X_i$ points. This cover has the following properties: First, every point $x$ of $X_i$ belongs to at most $d^R$ components of the cover, since each ball has cardinality $d^R$. Second, for every subset of $X_i$ with diameter less than $2R$, it is contained in at least one component of the cover. 
		
		By Lemma \ref{Lemma2}, the given sequence of kernels $(K_i:X_i\times X_i \to \R)_{i\in \N}$ satisfies $\sigma (P_x^R T_i P_x^R) \subseteq (-\infty,b_i]$, for every $i\in \N$ and $x \in X_i$, where $T_i$ is the operator associated to $K_i$. Define $K_i'$ as an operator in $\mathbb{B}( \ell^2 (X_i) )$ by $$T_i' \coloneqq T_i - b_i \sum_{j=1}^{|X_i|} \chi_{(0,\infty)}(P_{x_j}^R T_iP_{x_j}^R) ,$$ where $P_{x_j}^R:\ell^2(X_i)\to \ell^2(B_{R}(x_j))_0=\big\{f :B_{R}(x_j)\to \R; \sum\limits_{y\in B_R(x_j)} f(y)=0\big\}$ is the orthogonal projection, $\chi_{(0,\infty)}$ is the characteristic function of $(0,\infty)$ and the operator $\chi_{(0,\infty)}(P_{x_j}^R T_iP_{x_j}^R)$ exists by functional calculus.
		Now,
		
		$$\max_{x,y\in X_i} |K_i'(x,y) -K_i(x,y) | \leq \bigg| \langle b_i \sum_{j=1}^{|X_i|} \chi_{(0,\infty)}(P_{x_j}^R T_iP_{x_j}^R) \delta_x,\delta_y \rangle \bigg|  \leq b_i d^{R}.$$
		
		The last inequality comes from fact that each $x \in X_i$ is in at most $d^R$ components of the cover and $\chi_{(0,\infty)}(P_{x_j}^R T_iP_{x_j}^R)\leq 1$. We know that the sequence $(b_i)_{i\in \N}$ converges to zero, then take $C_{i,R}\coloneqq b_i d^{R}$. Thus, $|K_i' -K_i |_{max}\leq C_{i,R}$ that tends to zero when $i$ goes to infinity. Moreover, the sequence $ (K_i')_{i \in \N}$ approximates $\psi $ as $(K_i)_{i \in \N}$ also does.
		
		Using the fact that each subset of diameter less than $2R$ is in at least one of the components. We have  $P_{x}^{R}P_{x_k}^R=P_{x_k}^R P_{x}^{R}=P_{x}^{R}$, for at least one $k$. We prove that $K_i'$ is $R$-locally conditionally negative definite function calculating the spectrum of $P_{x}^{R} T_i' P_{x}^{R}$ for all $x\in X_i$. Notice that 
		
		\begin{equation*}
			\begin{split}
				P_{x}^{R} T_i' P_{x}^{R} & = P_{x}^{R} \bigg[ T_i - b_i \sum_{j=1}^{|X_i|}  \chi_{(0,\infty)}(P_{x_j}^R T_i P_{x_j}^R)\bigg] P_{x}^{R}\\
				& = P_{x}^{R} \bigg[P_{x_k}^R T_i P_{x_k}^R - b_i \sum_{j=1}^{|X_i|}  \chi_{(0,\infty)}(P_{x_j}^R T_i P_{x_j}^R)\bigg] P_{x}^{R}\\
				&= P_{x}^{R} \bigg[ P_{x_k}^RT_i P_{x_k}^R- b_i \chi_{(0,\infty)}(P_{x_k}^R T_i P_{x_k}^R) \bigg] P_{x}^{R}  - \\
				& \hspace{1cm}- P_{x}^{R}\bigg[ b_i \sum_{j\neq k } \chi_{(0,\infty)}(P_{x_j}^R T_i P_{x_j}^R)\bigg] P_{x}^{R}
			\end{split}
		\end{equation*}
		
		By construction, we can see that $\langle P_{x_k}^RT_i P_{x_k}^R- b_i \chi_{(0,\infty)}(P_{x_k}^R T_i P_{x_k}^R) f,f \rangle$ is less than or equal to zero and $-b_i \langle \sum_{j\neq k } \chi_{(0,\infty)}(P_{x_j}^R T_i P_{x_j}^R) f,f \rangle$ is also negative, for every $f \in \ell^2(B_R(x))_0$. Therefore, $K_i'$ is a $R$-locally conditionally negative definite function.\end{proof}

	By the lemma above, it remains to pick the appropriate sequence of radii $(R_i)_{i\in\N}$ such that the sequence of kernels $(K_i')_{i\in\N}$ induces an asymptotic coarse embedding into a Hilbert space.

	\begin{thm} \label{thmtopaTmenable}
		Let $\mathcal{X}=\{X_i\}_{i\in\N}$ be a sequence of finite graphs with bounded degree and $X$ the related space of graphs. If the coarse boundary groupoid $\partial G(X)$ is a topologically a-T-menable then the sequence of graphs $\mathcal{X}$ is asymptotically coarsely embeddable into a Hilbert space $\mathcal{H}.$
	\end{thm}
	\begin{proof}
		
		Let $\partial G(X)$ be a topologically a-T-menable groupoid, thus there is a continuous proper conditionally negative definite function $\psi: \partial G(X) \to \R$. Applying Lemma \ref{Lemma1}, there exists an approximation of $\psi$ by a sequence of symmetric limit-normalized functions $(K_i: X_i \times X_i \to \R )_{i\in\N}$ and increasing control functions $\rho_1^{K_i}, \rho_2^{K_i}: \R_+ \to \R$  such that
		\begin{equation} \label{controlofK}
			\rho_1^{K_i}(d(x_i,y_i)) \leq K_i(x_i,y_i) \leq \rho_2^{K_i}(d(x_i,y_i)) \hspace{0.4cm} \forall x_i,y_i \in X_i  \text{ and }i\in \N.
		\end{equation}
		For all $R>0$, the sequence of kernels $(K_i)_{i \in \N}$ satisfies Lemma \ref{Lemma2}, so we can apply Lemma \ref{Lemma3}, and obtain another sequence of kernels $(K_{i,R}')_{i\in \N}$ approximating $\psi$ such that $|K_{i,R}' -K_i |_{max}\leq C_{i,R}$, where the sequence of $(C_{i,R})_{i\in\N}$ converges to zero and the sequence $(K_{i,R}')_{i\in \N}$ is $R$-locally conditionally negative definite. Even more, it is limit-normalized.

		We need to exhibit an increasing sequence of $(R_i)_{i\in\N}$ such that a sequence of kernels is $(R_i)$-locally conditionally negative definite function. For that just take $K_i'=K_{i,R_i}'$ for  $$R_i\coloneqq \max \bigg\{R>0: C_{i,R} \leq \frac{1}{R} \bigg\}.$$

		We finish the proof with the existence of the control functions $\rho_1^{K_i'}$ and $\rho_2^{K_i'}$ for the sequence of kernels $(K_i')_{i\in\N}$.  Take $C=\sup\limits_{i \in \N} C_{i,R_i}$. The supremum exists since the sequence $(C_{i,R_i})_{i\in\N}$ goes to zero. By Equation \ref{controlofK} and the fact that $|K_i' -K_i |_{max}\leq C_{i,R_i}$, we have $$\rho_1^{K_i}(d(x_i,y_i))-C \leq K_i'(x_i,y_i) \leq \rho_2^{K_i}(d(x_i,y_i))+C .$$ 
		Thus, the control functions defined by $\rho_1^{K_i'}(x_i,y_i) =\rho_1^{K_i}(d(x_i,y_i)) - C$ is a lower bound and as $\rho_2^{K_i'}(x_i,y_i) =\rho_2^{K_i}(d(x_i,y_i)) + C$ is an upper bound of the sequence $(K_i'(x_i,y_i))_{i\in\N}$. 
		
		We conclude that the sequence of graphs $\{X_i\}_{i\in\N}$ is asymptotically coarsely embeddable into a Hilbert space.\end{proof}

	\subsubsection{Weak asymptotic coarse embeddability}

	In this section we analyze measurable version of a-T-menability for the boundary groupoid. Similarly to the amenable case, one easily observes that measurable a-T-menability of the coarse boundary groupoid $(\partial G(X),\hat{\mu}_\omega)$ along a non-principal ultrafilter (or even along every non-principal ultrafilter) is not a sufficient condition for the related sequence of graphs to be asymptotically coarsely embeddable into a Hilbert space. A counterexample can be easily constructed -- completely analogously to the amenable case from \cite{kaiser2019combinatorial} -- by placing a negligible geometric property (T) sequence of graphs (for instance, a box space of $\SL(3,\mathbb Z)$) near an asymptotically coarsely embeddable sequence (for instance, a box space of a free group).

	Our setup for this section is quite similar to the topological one, but the hypothesis of the measurable version is way weaker. We can not apply Tiezte's Theorem here to lift the function $\psi$ to a sequence of functions defined in the graphs. Instead of that, we define the sequence of kernels by the characterization of the groupoid via ultraproducts. More precisely, for a fixed non-principal ultrafilter $\omega \in\partial\beta \N$, given a conditional negative definite function $\psi: (\partial G (X),\hat{\mu}_\omega) \to \R$, there exists, by Proposition \ref{measufunc}, a sequence of maps $\psi_i$ from the graphs $ X_i\times X_i$ to the real numbers $\R$ that approximates $\psi$. We again view the lifts $\psi_i$ as kernels $K_i$ along the graph sequence.


	\begin{lemma}\label{controlmeasure}
		Let $\psi:({\partial G(X)}, \hat{\mu}_\omega) \to \R$ be a measurable conditionally negative definite function such that the sequence of kernels $(K_i)_{i\in \N}$ approximates $\psi$ along a non-principal ultrafilter $\omega \in \partial \beta \N$. Then, for all $\varepsilon>0$, there exists a function $\rho_2^\varepsilon:\R_+ \to \R$ such that $\lim\limits_{i\to\omega} \hat{\mu}_i(B_i) \leq \frac{\varepsilon}{2}$, where $$B_i=\{ (x_i,y_i) \in X_i \times X_i ; K_i(x_i,y_i)\geq \rho_2^\varepsilon(d(x_i,y_i))\}.$$
		
		Moreover, if $\psi$ is measurably proper then, for all $\varepsilon>0$, exists $\rho_1^\varepsilon:\R_+ \to \R$ such that $\lim\limits_{i\to\omega} \hat{\mu}_i(A_i) \leq \frac{\varepsilon}{2}$, where $$A_i=\{ (x_i,y_i) \in X_i \times X_i ; K_i(x_i,y_i)\leq \rho_1^\varepsilon(d(x_i,y_i))\}$$ and the functions $\rho^\varepsilon_1, \rho^\varepsilon_2$ go to infinity at infinity. In particular, $\lim\limits_{i\to\omega} \hat{\mu}_i (A_i\cup B_i )\leq \varepsilon$ and the maps $\rho_1^\varepsilon$, $\rho_2^\varepsilon$ are control functions for $(K_i)_{i\in \N}$ along the sequence of graphs $X_i\times X_i\setminus (A_i\cup B_i )$ with the induced length metric of $X_i\times X_i$. 
	\end{lemma}
	\begin{proof}
		Fix an $\varepsilon>0$ and non-principal ultrafilter $\omega\in\partial \beta \N$, recall the notation of Proposition \ref{measufunc}, where $H_{\omega,k,L}$ is approximated by the sequence of sets $H_{i,k,L}=\{ (x_i,y_i) \in X_i \times X_i;  d(x_i,y_i)=k, K_i(x_i,y_i)\in [L,L+1) \} $, i.e, $H_{\omega,k,L}= [H_{i,k,L}]_\omega$ and  $S_{\omega,k}=\{(x,y) \in \partial \beta X \times \partial \beta X ; d(x,y)=k  \}$, for $k\in \N$ and $L\in \Z$. Notice that $S_{\omega,k}\subseteq \overline{E_k}$ has finite measure and $S_{\omega,k} = \sqcup_{L=-\infty}^{\infty} H_{\omega,k,L}$. So, for any $k\in \N$, there is a $L^\varepsilon_k\in\N$ such that $$ \sum_{L=L^\varepsilon_k}^{\infty} \hat{\mu}_\omega(H_{\omega,k,L})\leq \dfrac{\varepsilon}{2^{k+1}}.$$

		For every $t \in \R_+$, define $\rho_2^\varepsilon(t)= L_{\lfloor t \rfloor}^\varepsilon$, where  $\lfloor t \rfloor$ is the integer part of $t$. Notice that $B_i=\{ (x_i,y_i) \in X_i \times X_i ; K_i(x_i,y_i)\geq \rho_2^\varepsilon(d(x_i,y_i))\} = \sqcup_{k\in \N} \sqcup_{L\geq L_k^\varepsilon} H_{i,k,L}$, where $k=d(x_i,y_i)$. Thus, 
		
		\begin{align*}
			\lim_{i\to\omega} \hat{\mu}_i (B_i) &= \lim_{i\to\omega} \hat{\mu} _i\bigg( \bigsqcup_{k\in \N}\bigsqcup_{L\geq L_k^\varepsilon}^\infty H_{i,k,L} \bigg)\\
			& \leq \sum_{k\in\N} \sum_{L=L^\varepsilon_k}^{\infty} \hat{\mu}_\omega(H_{\omega,k,L})\\
			&\leq \sum_{k\in\N} \dfrac{\varepsilon}{2^{k+1}}\\
			&= \frac{\varepsilon}{2}.
		\end{align*}

		Assume now that $\psi$ is a measurably proper function, i.e, for every $C>0$, the set $\{\gamma \in \partial G(X); \psi(\gamma) \leq C  \}$ has finite measure.  For a given $C>0$, exist $k'\in\N$ such that, for every $k\geq k'$, we have $ \hat{\mu}_\omega(\{\gamma \in S_{\omega,{k}}; \psi(\gamma) \leq C  \})\leq \varepsilon$, since we can decompose $\partial G(X)$ as $ \sqcup_{k\in \N} S_{\omega,k}$. We define the map $\rho_1^\varepsilon$ for every  positive real number $t$ as $$\rho_1^\varepsilon(\lfloor t \rfloor)= \max \bigg\{C>0 ; \hat{\mu}_\omega \bigg(\{\gamma \in S_{\omega,\lfloor t \rfloor}; \psi(\gamma) \leq C  \}\bigg)\leq \dfrac{\varepsilon}{2^{\lfloor t \rfloor+1}}\bigg\}.$$
		
		Furthermore, the set $A_i=\{ (x_i,y_i) \in X_i \times X_i ; K_i(x_i,y_i)\leq \rho_1^\varepsilon(d(x_i,y_i))\}$  can described as $ \sqcup_{k\in\N} \sqcup_{L=-\infty}^{\rho_1^\varepsilon(k)-1} H_{i,k,L}$, therefore 
		
		\begin{align*}
			\lim_{i\to\omega} \hat{\mu}_i(A_i) &= \lim_{i\to\omega} \hat{\mu}_i\bigg(\bigsqcup_k \bigsqcup_{L=-\infty}^{\rho_1^\varepsilon(k)-1} H_{i,k,L}\bigg) \\
			&\leq \sum_{k\in\N} \sum_{L=-\infty}^{\rho_1^\varepsilon(k)-1} \hat{\mu}_\omega(H_{\omega,k,L}) \\
			& \leq \sum_{k\in\N} \dfrac{\varepsilon}{2^{k+1}}\\
			&= \frac{\varepsilon}{2}. 
		\end{align*}

		By the construction, it follow that
		$$\lim\limits_{i\to\omega} \hat{\mu}_i (A_i\cup B_i )\leq \varepsilon$$
		Even more, the following inequality follows if we use the old length metrics of the graphs $X_i$
		$$\rho_1^\varepsilon(d(x_i,y_i)) \leq K_i(x_i,y_i)\leq  \rho_2^\varepsilon (d(x_i,y_i)),$$ for $x_i,y_i\in X_i\times X_i\setminus (A_i\cup B_i)$. 
		The old length metric $d_i$ of each graph $X_i$ is necessary here because after we remove the points in $A_i\cup B_i $ from the graph $X_i$, we might had disconnected the graph or increase the distance between points. Thus, the first inequality of the control functions $\rho_1^\varepsilon(d(x_i,y_i)) \leq K_i(x_i,y_i)\leq  \rho_2^\varepsilon (d(x_i,y_i))$ might not hold any longer. The second inequality still valid because $\rho_2^\varepsilon$ is monotonous and $d|_{X_i}(x_i,y_i) \leq d|_{X_i\times X_i\setminus A_i\cup B_i}(x_i,y_i)$, for all $x_i,y_i\in X_i\times X_i\setminus (A_i\cup B_i)$.

		Moreover, since $\psi$ is measurably proper, when $k$ tends to infinity then $\rho_1^\varepsilon(k)$ goes to infinity and so $\rho_2^\varepsilon(k)$ goes to infinity.\end{proof}

	Even in this situation, we can conclude that the sequence of graphs related to a measurably a-T-menable coarse boundary groupoid $\partial G(X)$ is ``almost asymptotically coarsely embeddable'' in the following sense.

	\begin{thm}\label{thmMeaaTmenable}
		Let $\mathcal{X}=\{X_i\}_{i\in\N}$ be a sequence of finite graphs with bounded degree and $X$ the related space of graphs. If the coarse boundary groupoid $(\partial G(X),\mu_\omega)$ is a measurably a-T-menable for every non-principal ultrafilter $\omega\in\partial\beta\N$, then, for every $\varepsilon>0$, there exist $\{Z_i\}_{i\in\N}\subset \{X_i\}_{i\in\N}$ such that $\lim\limits_{i\to\infty} {\mu}_i(Z_i)\leq \varepsilon$ and $\{X_i\backslash Z_i\}_{i\in\N}$ is asymptotically coarsely embeddable into a Hilbert space $\mathcal{H}$ when equipped with the restriction of the metric from $X$.
	\end{thm}

	\begin{proof}
		
		Given a fixed non-principal ultrafilter $\omega \in \partial \beta \N$ and a measurably proper conditionally negative definite function $\psi$ from $(\partial G(X),\mu_\omega )$  to $\R$. By Lemma \ref{measufunc}, we can approximate $\psi: \partial G(X) \to \R$ by $K_i : X_i \times X_i \to \R$ along the non-principal ultrafilter $\omega \in \partial \beta \N$.

		First, for some constant $a\geq0$, take the set $Z_i^a =\{x \in X_i : b_i(x) > a \}$, where $b_i(x)= \max \{0,\sigma (P_x^R T_i P_x^R)\}$ and $T_i$ is the operator induced by $K_i$. Since the sequence of kernels $(K_{i})_{i\in\N}$ approximates a conditionally negative definite function, the sequence $(b_i(x))_{i\in\N}$ tends to zero for every $x\in X_i$, then  $\lim\limits_{i \to \omega}\mu_i(Z_i^a) = 0$. So, for every $a>0$ and every $\varepsilon>0$, there is $n\in \N$ such that $\mu_i(Z_i^a)\leq \varepsilon$, for all $i\geq n$. Take $a_i \coloneqq \max \{ a \geq 0; \mu_i(Z_i^a)\leq \frac{1}{i}\}$.  In the same way, the sequence $(a_i)_{i\in\N}$ converges to zero when $i$ goes along the non-principal ultrafilter $\omega \in \partial \beta \N$. Therefore, the measure of $\mu_i(Z_i^{a_i})$ also tends to zero along $\omega \in \partial \beta \N$. For $Z_i'\coloneqq Z_i^{a_i}$ and by construction, 
		$$\bigcup_{x \in X_i\backslash Z_i'} \sigma \big( P_x^R  T_i P_x^R \big)  \in  (-\infty, a_i].$$
		
		In order to apply Lemma \ref{Lemma3}, we need to verify that the same holds but now under the projection $Q_x^R:\mathbb{B}(\ell^2(X_i\backslash Z_i')) \to \mathbb{B}(\ell^2(B_R(x)\cap X_i\backslash Z_i'))_0$, where $$\mathbb{B}(\ell^2(B_R(x)\cap X_i\backslash Z_i'))_0= \bigg\{f :B_{R}(x)\cap X_i\backslash Z_i'\to \R; \sum\limits_{y\in B_R(x_j)\cap X_i\backslash Z_i'} f(y)=0\bigg\}.$$ 
		
		Clearly, $P_x^R Q_x^R=Q_x^RP_x^R=Q_x^R$, for all $x\in X_i\backslash Z_i'$. Given $f \in\mathbb{B} (\ell^2(X_i\backslash Z_i'))$, we compute that:
		\begin{align*}
			\langle Q_x^R T_i Q_x^R f ,f \rangle &= 
			\langle Q_x^R P_x^R T_i P_x^R Q_x^R f ,f \rangle \\
			&=\langle  P_x^R T_i P_x^R Q_x^R f ,Q_x^R  f  \rangle.
		\end{align*} 
		
		In particular, $\langle  P_x^R T_i P_x^R f' ,  f'  \rangle  \leq a_i $, for all $ f' \in \mathbb{B}(\ell^2(B_R(x)\cap X_i\backslash Z_i'))_0$ then
		$$\bigcup_{x \in X_i\backslash Z_i'} \sigma \big( Q_x^R  T_i Q_x^R \big)  \in  (-\infty, a_i].$$
		
		For a fixed $R>0$, we can apply Lemma \ref{Lemma3}. So, there exist a sequence of kernels $(K_{i,R}')_{i\in \N}$ that is $R$-locally conditionally negative definite and $||K_{i,R}' - K_i ||_{max} \leq C_{i,R}.$ With the same argument done in Theorem 5.10, for $R_i\coloneqq \max \{R: C_{i,R} \leq \frac{1}{R} \}$, we obtain a sequence $(K_{i,R_i}': X_i\setminus Z_i'\times X_i\setminus Z_i'\to \R)_{i\in \N}$ of conditionally negative definite functions that approximates $\psi$ and an increasing sequence of real numbers $(R_i)_{i\in \N}$, we denote the functions only by $(K_i')_{i\in \N}$.

		We finish the proof with the existence of the control functions along the fixed ultrafilter. For a given $\varepsilon>0$, by Lemma \ref{controlmeasure}, there exist control functions $\rho_1^\varepsilon$ and $\rho_2^\varepsilon$ going to infinity at infinity just along the sets $(X_i\setminus Z_i') \times (X_i\setminus Z_i')\setminus(A_i\cup B_i)$. Take $$Z_i\coloneqq \{x_i \in X_i\setminus Z_i'; \exists y_i \in X_i\setminus Z_i', (x_i,y_i)\in A_i\cup B_i\}$$ that clearly has measure smaller than $\varepsilon$ along the ultrafilter.
		
		Thus, the sequence of graphs $\{X_i\setminus Z_i\}_{i\in \N}$ is asymptotically coarsely embeddable into a Hilbert space, since the restrictions of $K_i'$'s to $X_i\setminus Z_i$ preserve the conditionally negative definite condition and the control functions constrains holds along $X_i\setminus Z_i$ with the old graph metric $d|_{X_i}$.

		We've seen until this point that the sequence of graphs is asymptotically coarsely embeddable into a Hilbert space along the non-principal ultrafilter $\omega \in \partial \beta \N$. Now, we can use the similar arguments made before to extend the result for all natural numbers. For that, suppose that there is a $\varepsilon>0$ such that for all $\{Z_i\}_{i\in \N} $ with $\lim\limits_{i\to \infty} \mu_i(Z_i) \leq \varepsilon$, the sequence $\{X_i\setminus Z_i\}_{i\in \N}$ is not asymptotically coarsely embeddable into a Hilbert space $\mathcal{H}$. So, for all $R>0$, there is a $i_R\in\N$ such that $X_{i_R}\setminus Z_{i_R}$ is not asymptotically coarsely embeddable. The set $I_R=\{i_R; R>0\}$ is infinite, then there is a non-principal ultrafilter $\omega \in \partial\beta \N$ such that $I_R\subseteq \omega$. Thus, $\{X_i\setminus Z_i\}_{i\in \N}$ is not asymptotically coarsely embeddable into a Hilbert space, leading us in a contradiction.  
		
		This finish our proof, concluding that $\{X_i\backslash Z_i\}_{i\in\N}$ is asymptotically coarsely embeddable into a Hilbert space $\mathcal{H}$ in regard to the old length metric.	\end{proof}

One can observe two major differences to the corresponding result in the amenable case. First, instead of discarding an asymptotically negligible set from the graph sequence as in the definition of property almost-A, here we discard a subset of arbitrary small but positive measure. This is a consequence of the 
lifting result for measurably proper conditionally negative definite functions used in the proof of this theorem (Lemma \ref{controlmeasure}). This lifting seems indeed optimal since -- other than in the topological case -- a measurably proper function $\psi(x,y)$ can still remain small subsets of small measure even the distance $d(x,y)$ increases. This suggests that there should exist a sequence of graphs witnessing this phenomenon and therefore requiring $\eps > 0$ in the conclusion of the above theorem; however, we were unable to construct such a sequence directly and therefore leave it as an open question.
\begin{qu}
Does there exist a graph sequence such that $(\partial G(X),\hat{\mu}_\omega)$ is measurably a-T-menable, but does not satisfy the conclusion of the above theorem for $\eps = 0$? Notice that such an example cannot come from a sofic approximation of a group, as explained in the next section.
\end{qu}

Another important difference to the amenable case is that here the technique of the proof forces us to use the old metric even after discarding parts of the graphs as opposed to the induced metric. This naturally leads to the following question:
	
\begin{qu}
Is the sequence of graphs $\{X_i\backslash Z_i\}_{i\in\N}$ (considered \emph{with their own length metric}) asymptotically coarsely embeddable into a Hilbert space?
\end{qu}

Being unable to answer the above question, we are hesitant to decide what the right notion of being ``almost asymptotically coarsely embeddable into a Hilbert space'' should be. As an \emph{ad hoc} notion for the rest of the paper, we say that a sequence of graphs $\mathcal{X}=\{X_i\}_{i\in\N}$ is \emph{ weakly asymptotically embedable into a Hilbert space} if it satisfies the conclusion of Theorem \ref{thmMeaaTmenable}.


	\section{Applications to sofic approximations}

	In this section we apply the results presented in the previous section to the case where the sequence of graphs comes from a sofic approximation of a finitely generated group, thus linking the analytic properties of the group and the coarse geometry properties of the sequence of graphs. We take a finitely generated sofic group $\Gamma$ with sofic approximation $\mathcal{G}=\{X_i\}_{i\in\N}$. Let $X$ be the space of graphs constructed from the sofic approximation.
	
	First we apply Theorem \ref{mainthm} to rededuce the following result:
	\begin{cor}[{\cite[Theorem 5.6]{kaiser2019combinatorial}}]\label{amenableApli}
		Let $\Gamma$ be a sofic group with sofic approximation $\mathcal{G}=\{X_i\}_{i\in\N}$. Then, the sofic approximation $\{X_i\}_{i\in  \N}$ has property A on average if and only if the group $\Gamma$ is amenable.
	\end{cor}
	\begin{proof}
		Suppose that $\Gamma$ is amenable, then, the topological groupoid and thus the measured groupoid $\partial \beta X \rtimes \Gamma $ is amenable. According to \cite[Theorem 3.11]{alekseev2016sofic}, there is an almost everywhere isomorphism between $({\partial G(X)},\hat{\mu}_\omega)$ and $(\partial \beta X,\mu_\omega) \rtimes \Gamma $. Thus, by the Theorem \ref{mainthm}, the sofic approximation $\{X_i\}_{i\in \N}$ has property A on average.
		
		On the other hand, if the sofic approximation $\{X_i\}_{i\in  \N}$ has property A on average, then, the sofic approximation has property almost-A, and thus the coarse boundary groupoid $\partial G(X)$ is measurably amenable. For the reduction to the core $Z$, the groupoid $Z \rtimes \Gamma $ is also amenable with a $\Gamma$-invariant measure over $Z$. Therefore, the group $\Gamma$ is amenable.\end{proof}

	Mergin this corollary with all equivalences of Theorem \ref{mainthm} we obtain the following equivalences:
	
	$$\begin{array}{ccc}
		\mathcal{G}  \text{ is hyperfinite} &  & \\
		\Updownarrow &   &\\
		\mathcal{G}  \text{ has property almost-A }& \iff & \Gamma \text{ is amenable}\\
		\Updownarrow &   & \\
		\mathcal{G}   \text{ has property A on average }& & \\
	\end{array}$$
	\vspace*{\parskip}
	

	
	In the a-T-menable case, we could analogously apply Theorem \ref{thmMeaaTmenable} to deduce that every sofic approximation of an a-T-menable group becomes coarsely embeddable into a Hilbert space after removing an arbitrary small propotion of the vertices (but keeping the original metric). However, we can do better using the fact that a-T-menability of a sofic group gives us a conditionally negative definite function on the boundary that only depends on the group. Thus, the sofic approximation of a a-T-menable group is weakly asymptotically coarsely embeddable into a Hilbert space with $\varepsilon=0$.

	\begin{prop}\label{aTmeansoficapp}
		Let $\Gamma$ be a sofic finitely generated group with sofic approximation $\mathcal{G}=\{X_i\}_{i\in\N}$. Then, the group $\Gamma$ is a-T-menable if and only if there is a sequence of subgraphs $\{Z_i\}_{i\in\N} \subset\{X_i\}_{i\in\N} $ of the sofic approximation $\mathcal{G}$ such that $\lim\limits_{i\to \infty} \frac{|Z_i|}{|X_i|}=0$ and $\{X_i\setminus Z_i\}_{i\in\N} $ is asymptotically coarsely embeddable into a Hilbert space $\mathcal{H}$ with the old metric of the sofic approximation.
	\end{prop}
	\begin{proof}
		First, assume that the group $\Gamma$ is a-T-menable. From the proper conditional negative definite function $\psi: \Gamma \to \R$, we can induce a function $\psi'$ on the groupoid $\partial \beta X \rtimes \Gamma$ by $\psi'(x,g)= \psi (g)$ that is also proper conditional negative definite. In the measurable perspective, the function $\psi'$ is also measurably proper in $(\partial \beta X,\mu_\omega) \rtimes \Gamma$, for every non-principal ultrafilter $\omega \in \partial \beta \N$. Even more, the groupoid $(\partial \beta X,\mu_\omega) \rtimes \Gamma$ is almost everywhere isomorphic to $(\partial G(X), \hat{\mu}_\omega)$. Thus, the coarse boundary groupoid  $(\partial G(X), \hat{\mu}_\omega)$ is measurably a-T-menable with a function $\psi'$. Notice then, that the map $\psi'$ is the same in every fiber of $\partial G(X)$. So, the same phenomena that happens at Lemma \ref{controlmeasure} does not occur here, but we still need to discard a set of measure zero from the boundary due to the almost everywhere isomorphism of the groupoids and the measurable approximation.

		By Proposition \ref{measufunc}, for a fixed ultrafilter $\omega \in \partial \beta \N$, there is an approximation of $\psi': (\partial G(X), \hat{\mu}_\omega) \to \R$ via a sequence of functions $(K_i:X_i \times X_i \to \R)_{i\in \N}$. Using the same arguments from Theorem \ref{thmMeaaTmenable}, there exist a sequence of conditionally negative definite functions $(K_{i,R_i}': X_i\setminus Z_i\times X_i\setminus Z_i\to \R)_{i\in \N}$ that also approximates the function $\psi'$, when $R_i'= \max \{R>0; C_{i,R} \leq \frac{1}{R} \}$ is an increasing sequence of real numbers and the measure of $\mu_i(Z_i)$ tends to zero. We denote this approximation only by $(K_i')_{i\in \N}$.

		By the conditions of the proper conditionally definite negative function $\psi$, one can take control functions $\rho_1$, $\rho_2$ in a uniform way, as in the construction of Lemma \ref{Lemma1}. Remember that the approximation of $\psi$ by the sequence of kernels is a measurable approximation, so we still need to discard an negligible subsets along the sequence of graphs to also fix the conditional negative definiteness of the sequence. Thus, the sofic approximation is weakly asymptotically coarsely embeddable into a Hilbert space $\mathcal{H}$. 
		
		The converse follows from Theorem \ref{VadimTheo}, by the fact that there is an almost everywhere isomorphism between the coarse boundary groupoid associated to $\{X_i\}_{i\in\N}$ and the one related to the sequence $\{X_i\setminus Z_i\}_{i\in\N} $ and they are both a-T-menable.\end{proof}
	
Analogously to the general case of a graph sequence, we ask the following question:
\begin{qu}
Is the sofic approximation $\mc X' = \{X_i\backslash Z_i\}_{i\in\N}$ (considered \emph{with their own length metric}) asymptotically coarsely embeddable into a Hilbert space?
\end{qu}

\newcommand{\etalchar}[1]{$^{#1}$}



\end{document}